\documentclass[a4paper,11pt,twoside]{amsart}

\input{xy}
\xyoption{all}
\xyoption{poly}
\usepackage[all]{xy}
\usepackage{amsmath,amsthm, amsfonts, amssymb,stmaryrd,mathtools}
\usepackage{latexsym}
\usepackage{color}
\usepackage{microtype}
\usepackage{graphicx}
\usepackage{float}   
\usepackage[utf8]{inputenc}
\usepackage[small, bf, margin=90pt, tableposition=bottom]{caption}
\usepackage[T1]{fontenc}
\usepackage{enumerate}
\usepackage{verbatim}
\usepackage{hyperref}
\hypersetup{
    colorlinks=true, 
    linkcolor=red,  
    urlcolor=blue,  
    citecolor=blue
}

\newcommand\CC{{\mathbb{C}}}

\newcommand\NN{{\mathbb{N}}}
\newcommand\PP{{\mathbb{P}}}

\def\P{{\mathbb{P}}}

\def\L{\ensuremath{\mathbf L}}

\def\C{\ensuremath{\mathbb C}}
\def\A{\ensuremath{\mathbb A}}
\def\P{\ensuremath{\mathbb P}}

\def\A{{\mathcal A}}

\def\C{{\mathcal C}}
\def\D{{\mathcal D}}

\def\F{{\mathcal F}}

\def\I{{\mathcal I}}

\def\K{{\mathcal K}}
\def\L{{\mathcal L}}

\def\O{{\mathcal O}}

\def\T{{\mathcal T}}
\def\U{{\mathcal U}}

\def\PP{{\mathcal P}}

\def\Gr{{\ensuremath{Gr}}}
\def\d{\mathbf {d}}

\def\im{\operatorname{im}}

\providecommand{\abs}[1]{\lvert#1\rvert}

\newcommand{\bigslant}[2]{{\raisebox{.2em}{$#1$}\left/\raisebox{-.2em}{$#2$}\right.}}
\makeatletter
\newcommand{\xdashrightarrow}[2][]{\ext@arrow 0359\rightarrowfill@@{#1}{#2}}
\newcommand{\xdashleftarrow}[2][]{\ext@arrow 3095\leftarrowfill@@{#1}{#2}}
\newcommand{\xdashleftrightarrow}[2][]{\ext@arrow 3359\leftrightarrowfill@@{#1}{#2}}
\def\rightarrowfill@@{\arrowfill@@\relax\relbar\rightarrow}
\def\leftarrowfill@@{\arrowfill@@\leftarrow\relbar\relax}
\def\leftrightarrowfill@@{\arrowfill@@\leftarrow\relbar\rightarrow}
\def\arrowfill@@#1#2#3#4{%
  $\m@th\thickmuskip0mu\medmuskip\thickmuskip\thinmuskip\thickmuskip
   \relax#4#1
   \xleaders\hbox{$#4#2$}\hfill
   #3$%
}
\makeatother

\theoremstyle{definition} 
\newtheorem{theorem}{Theorem}[section]
\newtheorem{theorem*}{Theorem}
\newtheorem{corollary}[theorem]{Corollary}
\newtheorem{lemma}[theorem]{Lemma}
\newtheorem{ejem}[theorem]{Example}
\theoremstyle{definition}        
\newtheorem{proposition}[theorem]{Proposition}
\theoremstyle{definition}        
\newtheorem{definition}[theorem]{Definition}%[section]
\theoremstyle{definition}        
\newtheorem{remark}[theorem]{Remark}

\title{Logarithmic forms and singular projective foliations}

\author{Javier Gargiulo Acea}

\thanks{\textbf{This article will appear at Annales de l'Institut Fourier.} }

\keywords{Logarithmic forms - Singular projective foliations - Moduli spaces.}

\subjclass{14D20, 37F75, 14B10, 32S65.}

\begin{document}
 
%% English Abstract 
\begin{abstract}
In this article we study  polynomial logarithmic $q$-forms on a projective space and characterize those that define singular foliations of codimension $q$. Our main result is the algebraic proof of their infinitesimal stability when $q=2$ with some extra degree assumptions. We determine new irreducible components of the  moduli space of codimension two singular projective foliations of any degree, and we show that they are generically reduced in their natural scheme structure. Our method is based on an explicit description of the  Zariski tangent space of the corresponding moduli space at a given generic logarithmic form. Furthermore, we lay the groundwork for an extension of our stability  results to the general case $q\ge2$.
\end{abstract}

\maketitle

\tableofcontents

\pagebreak

\section{Introduction.}

 This article is concerned with the study of complex projective logarithmic forms of arbitrary degrees in the
setting of the theory of algebraic foliations. The problem which motivates this paper is the
analysis of the irreducible components of moduli spaces of algebraic projective singular foliations,
and the description of their geometry. For some classical results on this classification problem for codimension one foliations
we refer to e.g. \cite{OM}, \cite{CL} and \cite{J}. More in general, we recommend \cite{CP} and \cite{CPV} for some 
remarkable facts about the higher codimensional case.

Fix $q \in \NN$, $n,d \in \NN_{>q}$ and $m\in \NN_{\ge q+1}$. Also, set $\d=(d_1,\dots,d_m)$ such that $d=\sum_{i=1}^m d_i$. 
A polynomial logarithmic $q$-form of type $\d$ is a projective twisted differential form 
\begin{equation}\label{qlog}
\omega = \sum_{I\subset \{1,\dots,m\} \atop |I|=q}\lambda_I \hat{F}_I dF_{i_1}\wedge \dots \wedge dF_{i_q},
\end{equation}
where $\lambda \in \bigwedge^q \CC^m$ satisfies that its interior product by $\d$ vanishes (i.e., $i_{\d}(\lambda)=0\in \bigwedge^{q-1} \CC^m$), $F_1,\dots,F_m$ are homogeneous polynomials 
in $n+1$ variables of degrees $d_1,\dots,d_m$ and $\hat{F}_I = \prod_{j\notin I} F_j$ for every $I\subset \{1,\dots,m\}$. If we assume that $\lambda$ is totally decomposable (i.e., $\lambda = \lambda^1\wedge \dots \wedge \lambda^q$), then $\omega$ is an element of $H^0(\P^n,\Omega^q_{\P^n}(d))$ that also satisfies the Plücker’s decomposability equation \eqref{moduli} and the so called Frobenius’s integrability equation \eqref{moduli2} (see  Section \ref{basicdefs} for more details). These conditions ensure that the distribution associated to its kernel at each point determines a singular complex foliation of codimension $q$ on $\P^n$. 

In addition, the forms as in \eqref{qlog} are related to the classical sheaf of logarithmic forms in the following sense. With the notation above, if we set $\D_F$ as the divisor defined by the zero locus of $F = \prod_{i=1}^m F_i$, and assume it has simple normal crossings, then   
all the elements $\eta \in H^0(\P^n, \Omega^q_{\P^n}(\log \D_F))$ can be described in homogeneous coordinates by the formula
\begin{equation*}\label{qlogmero}
\eta = \frac{\omega}{F} =  \sum_{I\subset \{1,\dots,m\} \atop |I|=q}\lambda_I \frac{dF_{i_1}}{F_{i_1}}\wedge \dots \wedge \frac{dF_{i_q}}{F_{i_q}}.
\end{equation*}
Remember that the sheaf $\Omega_{\P^n}^{\bullet}(\log \D_F)$ is defined by all the rational forms $\eta$ such that $\eta$ and $d\eta$ have at most simple poles along $\D_F$. In other words, the global sections of the sheaf of logarithmic forms over $\D_F$ are in 1-1 correspondence with the vectors $\lambda\in \bigwedge^q \CC^m$ such that $i_{\d}(\lambda)=0$. See Proposition \ref{logqforms} for more details. 

On the other hand, denote by $\F_q(d,\P^n)\subset \P H^0(\P^n,\Omega^q_{\P^n}(d))  $ the moduli space of algebraic projective singular foliations of degree $d$ on $\P^n$, which corresponds to the algebraic space of twisted $q$-forms of total degree $d$ satisfying the announced equations \eqref{moduli} and \eqref{moduli2}.

We consider $\rho$ as a rational map parametrizing those logarithmic $q$-forms of type $\d$ that define foliations (see Definition \ref{defnaturalparametrization}), and we set $\L_q(\d,n)\subset \F_q(d,\P^n)$ as the Zariski closure of its image. We refer to this projective irreducible variety as the logarithmic variety associated to the partition $\d$. Our purpose is to show that these logarithmic varieties are irreducible components of the corresponding moduli spaces of foliations. 
In this paper, we solve this problem for $q=2$ and with some condition among the vector of degrees $\d$ (see the 2-balanced assumption from Definition \ref{defbalanced}). Furthermore, we expect that, up to some technical details, our results and proofs can be adapted for any $q\le n-2$ at least assuming an extended q-balanced condition. When $q =n-1$ there is no hope to establish a stability result even for a generic logarithmic form. In this case the corresponding moduli space coincides with $\P H^0(\P^n,\T \P^n)$.

In order to put our work in context, we shall explain some related articles. It was proved in \cite{OM} by Omegar Calvo Andrade that the logarithmic varieties $\L_1(\d,n)$ are irreducible components of $\F_1(d,\P^n)$, 
and so this type of forms satisfy a sort of stability condition among all the integrable forms (see integrability condition \eqref{moduli} for $q=1$). Actually, his article is essentially based on analytical and topological methods, and the main results are proved for logarithmic foliations on a general complex manifold $M$ with the assumption $H^1(M,\CC)=0$. 
Later, in \cite{CGM} it was developed a new proof of the stability result for polynomial logarithmic
one-forms in projective spaces. Here the methods are completely algebraic and provide further information about the irreducible components $\L_1(\d,n)$. For example, it is deduced that the scheme $\F_1(d,\P^n)$ is generically reduced along these components. 

Moreover, in \cite{CPV} the authors proved that the varieties $\L_q(\d,n)$ are irreducible components when $m=q+1$ (lower possible value for $m$). In this case, the corresponding differential $q$-forms determine projective foliations which are tangent 
to the fibers of quasi-homogeneous rational maps. From now, we refer to this type of forms as rational $q$-forms of type $\d$. It is also remarkable that there are not many known irreducible components of the moduli space $\F_q(d,\P^n)$ for  general $q$. See for instance the introductions of \cite{CP} and \cite{CPV}.

In some sense, this article is concerned with obtaining a common generalization of the definitions and results from \cite{OM} and \cite{CPV}, using the same algebraic methods as in \cite{CGM}.
Notice that formula \eqref{qlog} coincides with that given for a general rational $q$-form of type $\d$ in \cite{CPV} when $q=m-1$, 
and with that given for a polynomial logarithmic 1-form in \cite{OM} when $q=1$.

The article is organized as follows. In Section \ref{primera} we present a brief summary of the objects involved in the definitions of $\F_q(d,\P^n)$, $\L_q(\d,n)$ and the corresponding parametrization map, with a particular attention to the case $q=2$. We suggest to consult Definitions \ref{pollogformsdef1} and \ref{defnaturalparametrization}, and also Propositions \ref{pollogformsprop1} and \ref{characterizationvectors} that contain our definition of logarithmic q-forms of type $\d$, and a characterization of those that determine singular projective foliations of codimension q.

Since our methods are based on Zariski tangent space calculations, in Section \ref{primera}  we also present a description of these spaces for $\F_2(\d,\P^n)$ and the space of parameters $\PP_2(\d)$ of the natural parametrization $\rho$. In addition, a formula for the derivative of the  parametrization map is given at the end of 
the section. Our main tool is based on the comparison of the image of this derivative with the tangent space of $\F_2(d,\P^n)$ at a generic logarithmic form of type $\d$. See our Subsections \ref{modulizariskitangent} and \ref{The derivate of the natural parametrization} for a complete treatment of these tasks. 

Our Section \ref{infstability2logforms} contains a development of the main result: Theorem \ref{main}, which establishes that the logarithmic varieties $\L_2(\d,\P^n)$ 
are irreducible components of the moduli space $\F_2(d,\P^n)$. In addition, we deduce that these components are generically reduced according to its induced scheme structure. 
Our result assumes that $m,n >3$ and that the vector of degrees $\d$ is 2-balanced, that is
$$d_i + d_j < \sum_{k\ne i,\ne j} d_k \hspace{0.3cm} \forall i,j \in \{1,\dots,m\}.$$
The formal statement of the theorem is the following:
    \begin{theorem*}
    \textit{Fix natural numbers $n,m\in \NN_{> 3}$, and a 2-balanced vector of degrees $\textbf{d}=(d_1,\dots,d_m)$ with $d=\sum_i d_i$. The variety
    $\L_2(\textbf{d},n)$ is an irreducible component of the moduli space $\F_2(d,\P^n)$. 
    Furthermore, the scheme $\F_2(d,\P^n)$ is generically reduced along this component, in particular at the points of} $\rho(\U_2(\d)$).
    \end{theorem*}
The proof of Theorem \ref{main} is supported on Proposition \ref{surjectivity} (surjectivity of the derivative of the natural parametrization). This method had been originally used in \cite{CP} and \cite{CPV}, and was also applied in \cite{CGM}. However, the proof of this proposition is quite technical, and it is obtained through various steps and lemmas, including certain results of independent interest. See for instance Steps 1 to 7.

Finally, we would like to observe the followings facts about the general case $q\ge2$.  The definition of our natural  parametrization $\rho$, the varieties $\L_q(\d,n)$  and  some of our lemmas are still valid for larger $q$ (see Subsection \ref{lemmassection}). However, most of the steps of the proof of Proposition \ref{surjectivity}  are quite technical, and the combinatorics behind them is hard to extend for $q>2$. See also Remark \ref{casogeneral} for more details.

\subsection*{Acknowledgements}

The author is grateful with Fernando Cukierman for suggesting  the problem and for his further valuable help. Also gratitude is due to Federico Quallbrunn, Cesar Massri and the anonymous Referee for 
their useful comments and corrections. 

The content of this paper is part of the author’s doctoral thesis, for the degree of
Doctor de la Universidad de Buenos Aires.

    \begin{remark}
    After completing and having posted this article we learned about the article \cite{CL2} by D. Cerveau and A. Lins Neto.
    Some of our results for logarithmic $q$-forms turn out to be contained in this paper, but the methods of proof are completely different. In particular, the authors obtain a 
    stability result for the general case $1\le q \le n-2$, after reducing the problem to the case of foliations of dimension two (see \cite[Theorems 5 and 6]{CL2}).
    We would like to observe that our methods, based on computing the Zariski tangent space at a generic point, are completely algebraic and provide further information in the
    particular case of logarithmic foliations of codimension two. Especially, the fact that the moduli space of projective singular foliations results generically reduced along 
    the logarithmic irreducible components that we obtained.
    \end{remark}

\section{Definitions and constructions.}\label{primera}

\begin{flushleft}

In this article we shall use the following basic notation concerning polynomials and forms.

\medskip

$S^{(n)} = \CC[x_0, \dots, x_n]$ for the graded ring of polynomials with complex coefficients in $n+1$ variables. When $n$ is understood we write $S^{(n)} = S$.

\medskip

$S_{d}^{(n)} = H^0(\P^n,\O_{\P^n}(d))$ for  the space of homogeneous polynomials of degree $d$ in $n+1$ variables. 
When $n$ is understood we denote $S_{d}^{(n)}=S_d$.

\medskip

$\Omega^q_{X}(\L) = \Omega_X^q \otimes \O(\L)$ for the sheaf of differential forms twisted by a line bundle $\L$ on $X$.

\medskip

$i_v:\Omega^{\bullet}_{X} \rightarrow \Omega^{\bullet -1}_{X}$  for the interior product or contraction with a vector field $v$ on a variety $X$. 

In addition, if $m,q \in \NN$, $\d = (d_1,\dots,d_m)\in \CC^m$ and $\lambda \in \bigwedge^q\CC^m$, we 
also denote by $i_{\d}(\lambda) \in \bigwedge^{q-1}\CC^m$ the interior product of $\lambda$ by $\d$, that is

\begin{multline*}
 i_{(d_1,\dots,d_{m})}\left(\sum_{I:|I|=q}\lambda_I e_{i_1}\wedge \dots \wedge e_{i_q} \right ) \\ = \sum_{I:|I|=q }\sum_{j=1}^q \lambda_I(-1)^{j+1} d_{i_j} e_{i_1}\wedge \dots \wedge \hat{e}_{i_j} \wedge \dots \wedge e_{i_q}, 
 \end{multline*}
for $\{e_1,\dots,e_m\}$ the canonical base of $\CC^m$.

\end{flushleft}

\subsection{Projective singular foliations and logarithmic forms.}\label{basicdefs}
We write $\F_q(d,\P^n)$ for the moduli space of singular projective foliations of codimension $q$ and 
degree $d$. Recall from \cite{DM} (or see also \cite{CPV}) that this space is naturally  described 
by projective classes of twisted projective forms $\omega \in H^0(\P^n,\Omega^q_{\P^n}(d))$ which satisfy both the Plücker's decomposability condition
 \begin{equation}\label{moduli}
 i_v(\omega)\wedge \omega =0 \hspace{0.5cm} \forall v \in \bigwedge^{q-1}\CC^{n+1}
 \end{equation}
and the so called Frobenius's integrability condition 
 \begin{equation}\label{moduli2}
 i_v(\omega)\wedge d\omega=0 \hspace{0.5cm} \forall v \in \bigwedge^{q-1}\CC^{n+1}.
 \end{equation}
 The elements like $v$ can be considered as local 
frames on the affine cone over $\P^n$, or alternatively as rational multi-vector fields. 

The first equation ensures that $\omega$ is locally decomposable outside its singular set, so this section of $\Omega^q_{\P^n}(d)$ belongs to the corresponding Grassmannian space at that points. The second equation is a condition to guarantee that the singular distribution associated to the kernel of $\omega$ is also integrable. See for instance \cite[Proposition 1.2.2]{DM}.
We also recall that the singular set of the foliation induced by $\omega$  corresponds to its  vanishing points 
$S_{\omega} =\{p\in \P^n : \omega(p)=0\}.$

\begin{remark}
Each element $\omega \in H^0(\P^n,\Omega^q_{\P^n}(d))$ can be represented in homogeneous coordinates by  
homogeneous affine q-forms of total degree $d$ like
\begin{equation*}
\omega = \sum_{I=\{i_1\dots,i_q\}} A_I(z)dz_{i_1}\wedge \dots\wedge dz_{i_q},
 \end{equation*}
where $\{A_I\}_I\subset S_{d-q}$ are selected in order to satisfy the so called \textit{descend} condition
\begin{equation}\label{descensoq}
i_R(\omega)=0 \in H^0(\CC^{n+1},\Omega_{\CC^{n+1}}^{q-1}).
 \end{equation}
Here $R$ denotes the radial Euler field $\sum_i z_i\tfrac{\partial}{\partial z_i}$. 
\end{remark}
 
\begin{remark} 
 For $q=2$ the decomposability equation \eqref{moduli} is slightly simpler than in the
general case because is equivalent to
$i_v(\omega\wedge \omega) =0 \,\, \forall v \in \CC^{n+1},$
so it can be replaced by 
\begin{equation}\label{moduli3}
\omega \wedge \omega =0. 
\end{equation}
\end{remark} 
 
In conclusion, we define $\F_q(d,\P^n)$, the algebraic space of codimension $q$ singular foliations on $\P^n$ of degree $d\ge q$, as the following set
\begin{multline*}
\Bigg\lbrace [\omega]  \in \P H^0(\P^{n},\Omega^q_{\P^{n}}(d)):  
\omega \, \mbox{satisfies}\,  \eqref{moduli} \,\mbox{,} \, \eqref{moduli2}\,  \mbox{and}\, \ensuremath{codim}(S_\omega) \ge 2  \Bigg\rbrace. \end{multline*}
Thus each $[\omega]\in \F_q(d,\P^n)$ determines a singular holomorphic foliation on $\P^n$ whose leaves are of codimension $q$, i.e., 
a regular holomorphic foliation outside the singular set $S_{\omega}$. Moreover,  the degree of the divisor of tangencies of such leaves with a generic linearly embedded $\P^q$ in $\P^n$
is in this case $d-q-1$ (sometimes also used to refer to the degree of the corresponding foliation). 

On the other hand, we want to define correct formulas for those polynomial logarithmic $q$-forms in $H^0(\P^n,\Omega^q_{\P^n}(d))$ that define foliations according to the previous mentioned equations. 

First, we recall the principal definitions and properties of the well-known  sheaf of meromorphic logarithmic forms and its classical residues in our particular case of interest. For a more extensive development of this task see \cite{DEL}. 

Let $\D = \sum_{i=1}^m (F_i=0)$ be a projective divisor defined by homogeneous polynomials $F_i \in S_{d_i}$ for $i=1,\dots,m$. The sheaf $\Omega_{\P^n}^q (\log \D)$ of meromorphic logarithmic $q$-forms is defined as a subsheaf of the sheaf $\Omega^q_{\P^n}(\ast \D)$ of meromorphic forms with arbitrary poles along $\D$ according to 
$$\Omega^q_{\P^n}(\log \D) = \{\alpha \in \Omega^q_{\P^n}(\ast \D) : \alpha \, \mbox{and}\, d\alpha \, \mbox{have simple poles along}\, \D \}.$$
If $\D$ has simple normal crossings, it is a well known fact that  
$$(\Omega_{\P^n}^{\bullet}(\log \D),d) \hookrightarrow (\Omega_{\P^n}^{\bullet}(\ast \D),d)$$ determines a subcomplex  and a quasi-isomorphism (see \cite[Proposition 4.3]{STE}). 
We consider the following usual  filtration
\begin{equation}\label{filtration}
W_k(\Omega^q_{\P^n}(\log \D)) = 
\begin{cases}
0 & \mbox{if } k<0\\
\Omega^{q-k}_{\P^n} \wedge \Omega^k_{\P^n}(\log \D) & \mbox{if } 0\le k\le q\\
\Omega^q_{\P^n}(\log \D) & \mbox{if } k\ge q
\end{cases},
\end{equation}
and use the notation
\begin{itemize}
    \item $X_i= (F_i=0)$ for $i=1,\dots,m$,
   \item $X_I = X_{i_1} \cap \dots \cap X_{i_k} \hspace{0.3cm}\mbox{for }I=\{i_1 \dots i_k\}\subset \{1,\dots,m\}$,
   \item $\D(I) = \sum_{j\not\in I} X_I\cap X_j $ as a divisor on $X_I$,
   \item $j_I = X_I \hookrightarrow \P^n$ ,
      \item $X_{\D}^k = \coprod_{I: \abs{I}=k} X_I$,
   \item $j_k = X_{\D}^k \hookrightarrow \P^n$.
\end{itemize}

The desired formula for polynomial logarithmic forms will be supported on the characterization of the global sections of the sheaf  $\Omega_{\P^n}^q(\log \D)$. Our arguments will rely on the following known constructions and results regarding this classical sheaf of forms.

\begin{proposition}\label{propslogforms} Assume that $\D$ has simple normal crossings.

\noindent a) There is a well defined residue map $$Res_I: \Omega_{\P^n}^{\bullet}(\log \D)\longrightarrow \Omega_{X_I}^{\bullet}(\log \D(I))[-k]$$ for each multi-index 
$I$ with $\abs{I}=k$.

\noindent b) The residue map restricts to
$$Res_I: W_k(\Omega_{\P^n}^\bullet(\log \D) \longrightarrow (j_I)_*( \Omega_{X_I}^\bullet[-k]),$$
is surjective and also well defined on the quotient $Gr_k^W(\Omega_{\P^n}^\bullet(\log \D)).$

\noindent c)    The total residue map 
   $$Res_k := \bigoplus_{I: \abs{I}=k} Res_I : Gr_k^W(\Omega_{\P^n}^\bullet(\log \D)) \longrightarrow (j_k)_*(\Omega_{X_{\D}^k}^\bullet[-k])$$
   is an isomorphism.
\end{proposition}
\begin{proof}
 See \cite[Lemma 4.6]{STE}, or \cite{DEL} for an extended overview. 
 \end{proof}
Now we are able to compute $H^0(\P^n,\Omega_{\P^n}^q(\log \D))$. 
 
\begin{lemma}\label{cohomfiltration}
 With the notation above, assume $q<\mbox{min}\{m,n-1\}$. Then the following computations hold:
 \begin{enumerate}[i)]
 \item $H^0(\P^n,W_k(\Omega^q_{\P^n}(\log \D)))= 0 \hspace{0.3cm} \forall k=0,\dots,q-1,$
 \item  $H^1(\P^n, W_k(\Omega^q_{\P^n}(\log \D))) = 0 \hspace{0.3cm} \forall k=0,\dots,q-2,$
\item $H^1(\P^n, W_{q-1}(\Omega^q_{\P^n}(\log \D)))$ has a natural injective map to $\bigwedge^{q-1}\CC^m$.
 \end{enumerate}
\end{lemma}

\begin{proof}
 First, recall that the Hodge numbers associated to the classical projective space are
 $h_{\P^n}^{p,q} = \delta_{pq}$ (Kronecker delta).
 Due to the normal crossings condition and the Lefschetz hyperplane theorem (see for instance \cite[Section 3.1]{LA}) it is also true that
 \begin{equation*}
h_{X_I}^{p,q} = \delta_{pq} \hspace{0.3cm} \forall p,q,I : p+q < n-|I|,
  \end{equation*}
for every complete intersection subvariety $X_I$. Now, we proceed by induction on the number $k\ge0$. 
The base case is trivial. Then, the results i) and ii) are immediate consequences of considering the long exact sequence 
in cohomology associated to the exact sequence
\begin{equation}\label{exactsequencefilt}
0 \longrightarrow W_{k-1}(\Omega^q_{\P^n}(\log \D)) \longrightarrow W_{k}(\Omega^q_{\P^n}(\log \D)) \longrightarrow (j_k)_*(\Omega^{q-k}_{X_{\D}^k})\longrightarrow 0,
\end{equation}
where the exactness follows from Proposition \ref{propslogforms}.
 Finally,  consider the long sequence on cohomology for  $k=q-1$, that is
\begin{multline*}
0\rightarrow H^1(\P^n,W_{q-1}(\Omega^q_{\P^n})) \rightarrow \bigoplus_{I:|I|=q-1} H^1(X_I,\Omega^1_{X_I}) \\
\rightarrow H^2(\P^n,W_{q-2}(\Omega^q_{\P^n}(\log \D)))\dots \,\, ,
 \end{multline*}
 and use the knowledge of the Hodge number $h_{X_I}^{1,1}$ for each $I$ of size $q-1$ to show our assertion iii).
\end{proof}

With a similar proof and the same inductive argument, we can also state the following result.

\begin{corollary}
If we assume that $\D$ has simple normal crossings and $q< \mbox{min}\{m,n-j\}$, then  
$$H^j(\P^n,W_k(\Omega_{\P^n}^q(\log \D))) = 0 \hspace{0.3cm} \forall k=1,\dots,q-j-1.$$
\end{corollary} 
 
 \begin{proposition}\label{logqforms}
 Assume $q< \mbox{min}\{m,n-1\}$. Let $F_1,\dots, F_m$ be homogeneous polynomials with simple normal crossings and respective degrees $d_1,\dots,d_m$, defying a divisor $\D$ as before. For every global 
 section $\eta \in H^0(\P^n,\Omega_{\P^n}^q(\log \D)),$ there exist unique constants $\lambda=\{\lambda_I\}_{I:|I|=q}\in \bigwedge^q\CC^m$ such that $\eta$  can be written in homogeneous coordinates as
 $$\eta  = \sum_{I\subset \{1,\dots,m\} \atop |I|=q} \lambda_I \frac{dF_{i_1}}{F_{i_1}} \wedge \dots \wedge \frac{dF_{i_q}}{F_{i_q}},$$
 where also $\lambda$ satisfies $i_{\d}(\lambda)=0$ to ensure that $i_R(\eta)=0$. 
 Recall that $i_{\d}$ denotes the interior product by the constant vector $\d=(d_1,\dots,d_m)$.
 \end{proposition}

\begin{proof}
We use once more the exact sequence of sheaves \eqref{exactsequencefilt} considered at the proof of Lemma \ref{cohomfiltration}, but now in the particular case $k=q$, and we get
$$0 \longrightarrow W_{q-1}(\Omega^q_{\P^n}(\log \D)) \longrightarrow \Omega^q_{\P^n}(\log \D) \xrightarrow{Res_q} (j_q)_*(\O_{X_{\D}^q})\longrightarrow 0.$$
The result follows from considering again the associated long exact sequence in cohomology, and applying the previous lemma to deduce the injectivity of the total residue map $Res_q$ on global sections. 
Notice that the normal crossing assumption ensures that the image of the total residue
map described in homogeneous coordinates can be considered as a subset of  $\CC^{\binom{m}{q}} \simeq \bigwedge^q\CC^m$. Hence, the residues of the form $\sum\lambda_I \frac{dF_{i_1}}{F_{i_1}} \wedge \dots \wedge \frac{dF_{i_q}}{F_{i_q}}$ are the numbers $(\lambda_I)$. In addition, the kernel of the first connection morphism of the long exact sequence is exactly determined by the descend condition among the previous form, that is $i_{\d}(\lambda)=0$. This equation characterizes the image of $Res_q$ and implies our statement.
\end{proof}
 
\begin{remark}
 In fact, last proposition can be proved for a  general complex projective algebraic variety $X$ of dimension $n$ with the assumption
 $$h_X^{p,0} = dim(H^0(X,\Omega_X^p))= 0,$$ for   $1\le p \le n-1.$
\end{remark}

\begin{corollary}\label{jouanoloulemmasnc}
 Select $\lambda \in \bigwedge^q \CC^m$ and homogeneous polynomials $F_i \in S_{d_i}$ having simple normal crossings. Then 
 $\sum\lambda_I \frac{dF_{i_1}}{F_{i_1}} \wedge \dots \wedge \frac{dF_{i_q}}{F_{i_q}}=0$ if and only if $\lambda =0$.
\end{corollary} 
 
 The last statement is a sort of extension of \cite[Lemma 3.3.1]{J} to the case of higher degree logarithmic forms, with a more restrictive hypothesis.
 
 \begin{corollary}\label{logKoszul}
For every simple normal crossings divisor $\D$ set as before we have $$H^0(\P^n,\Omega_{\P^n}^q(\log \D)) \simeq \{\lambda \in \bigwedge^q \CC^m : i_{\d}(\lambda)=0 \}.$$
Moreover, notice that for each $q\in\NN_{< m}$ the right set corresponds to the degree $q$ cycles of the Koszul complex associated to the vector $\d\in \CC^m$, i.e., the exact complex 
$$\K(\d): 0\rightarrow \bigwedge^m \CC^m  \xrightarrow{i_{\d}} \bigwedge^{m-1} \CC^m \xrightarrow{i_{\d}} \dots \xrightarrow{i_{\d}} \bigwedge \CC^m \xrightarrow{i_{\d}}  \CC \rightarrow  0.$$
\end{corollary}

\begin{definition}
Fix natural numbers $d$ and $m$. A \textit{partition} of $d$ into $m$ parts is an m-tuple of degrees $\d= (d_1,\dots,d_m)\in \NN^m$
 satisfying $\sum_{i=1}^m d_i =d$. 
\end{definition}

For future convenience, we will use the following notation. 
 
\begin{definition}\label{notationpolynomials}
For $F_i \in S_{d_i}$ as above we denote
$$\underline{F} = (F_1, \dots, F_m), $$
$$F = \prod_{j=1}^m F_j, \ \ \ \ \hat{F}_i =  \prod_{j \ne i}  F_j , \ \ \ \  
\hat{F}_{i j} =  \prod_{k \ne i, k \ne j}  F_k , \ (i \ne j),$$ 
or, more generally, for $I \subset \{1, \dots, m\}$ we write
$\hat{F}_I =  \prod_{j \notin I} F_j.$
\end{definition}

 \begin{definition}\label{pollogformsdef1}
 Fix $\d$ as a partition of $d\in \NN_{>q}$ into $m$ parts, and assume $m\ge q+1$. A twisted projective differential $q$-form  $\omega \in H^0(\P^n,\Omega^q_{\P^n}(d))$ is a polynomial logarithmic q-form of type $\d$  
 if it is defined by the following formula
 \begin{equation}\label{logq}
 \omega = F  \left (\sum_{I:|I|=q}\lambda_I \tfrac{dF_{i_1}}{F_{i_1}}\wedge \dots \wedge \tfrac{dF{i_q}}{F_{i_q}} \right ) =  \sum_{I:|I|=q} \lambda_{I} \hat{F}_I dF_I ,
 \end{equation}
 where $F_i \in S_{d_i}$ and $\lambda \in \bigwedge^q \CC^m $  satisfies 
 $i_{\d}(\lambda) = 0$. 
\end{definition} 
 
\begin{remark}
 From Corollary \ref{logKoszul} we deduce that for every logarithmic $q$-form $\omega$ as in \eqref{logq}, there exist $\gamma\in \bigwedge^{q+1} \CC^m$ such that $i_{\d}(\gamma) = \lambda$ and 
 $$\omega = i_R \left ( \sum_{J:|J|=q+1} \gamma_{J} \hat{F}_J dF_J \right ).$$
 Hence the above formula could be an alternative definition for logarithmic $q$-forms of type $\d$. See also \cite{CPV} to compare our formulas with the definition presented there for rational $q$-forms.
\end{remark}

At first, we want to determine when this type of forms define foliations of codimension $q$, i.e., when the forms as in \eqref{logq} satisfy the equations of $\F_q(d,\P^n)$. For simplicity, we start with the Plücker's decomposability condition \eqref{moduli} (or equivalently  \eqref{moduli3}) for $q=2$.

\begin{proposition}\label{pollogformsprop1}
 If $\omega$ is a logarithmic 2-form of type $\d$ and the polynomials involved have simple normal crossings, then the following conditions are equivalent:
 \begin{enumerate}[i)]
  \item $\omega \wedge \omega = 0$
  \item $\lambda \wedge \lambda=0$.
 \end{enumerate}
\end{proposition}

\begin{proof}
Notice that 
 $$ \omega \wedge \omega = F^2  \left ( \sum_{i\ne j \ne k \ne l} (\lambda \wedge \lambda)_{ijkl} \tfrac{dF_i}{F_i}\wedge \tfrac{dF_j}{F_j}\wedge \tfrac{dF_k}{F_k}\wedge \tfrac{dF_l}{F_l} \right ).$$
 The result then follows from Corollary \ref{jouanoloulemmasnc}.
\end{proof}

\begin{remark}
For $omega$ as above, the condition $\lambda \wedge \lambda = 0$ implies that there exist  $\lambda^1,\lambda^2 \in \CC^m$ such that $\lambda = \lambda^1\wedge \lambda^2$.
 In this case, if we consider the meromorphic logarithmic forms defined by
 $$\eta _j = \sum_{i=1}^m (\lambda^j)_i \tfrac{dF_i}{F_i} \hspace{0.3cm}j=1,2,$$
then the meromorphic logarithmic 2-form $\eta = \tfrac{\omega}{F} = \eta_1 \wedge \eta_2$ is globally decomposable.
\end{remark}

\begin{remark}\label{decompositionqforms}
If we consider a totally decomposable $q$-vector $\lambda = \lambda^1 \wedge \dots \wedge \lambda^q \in \bigwedge^q \CC^m$ with $i_{\d}(\lambda) =0$,
 then for every selection of homogeneous polynomials, the induced logarithmic $q$-form of type $\d$ as in \eqref{logq} satisfies the Plücker's decomposability equation 
 \eqref{moduli}. Using the above notation, the reason is that the meromorphic $q$-form
 $\eta = \tfrac{\omega}{F} = \eta_1 \wedge \dots \wedge \eta_q$ is globally decomposable.
\end{remark}

Let us recall that the grassmannian space $\Gr(q,\CC^m)$ of q-dimensional subspaces of $\CC^m$ can be considered as a projective algebraic subset 
of $\P (\bigwedge^q \CC^m)$ via the Plücker embedding
\begin{align*}
 \iota: \Gr(q,\CC^m) &\longrightarrow \P(\bigwedge^q \CC^m)\\
 \mbox{span}(\lambda^1,\dots,\lambda^q) &\longmapsto [\lambda^1\wedge\dots \wedge \lambda^q].
\end{align*}

With a slight abuse of notation, we only write  $\lambda$ or $\lambda^1\wedge \dots \wedge \lambda^q$ for the elements of  $\Gr(q,\CC^m)$ and  
$(\lambda_{I})$ for its corresponding antisymmetric coordinates.

\begin{proposition}\label{characterizationvectors}
For  $\d$ as above,  $\lambda  \in \Gr(q,\CC^m)$ such that $i_{\d}(\lambda)=0$ and $(F_i)_{i=1}^m\in \prod_{i=1}^m S_{d_i}$, the logarithmic q-form of type $\d$
 $$\omega = \sum_{I:|I|=q} \lambda_{I}\hat{F_{I}}dF_I $$
  is a twisted projective form of total degree $d$, i.e., $\omega \in H^0(\P^n,\Omega^q_{\P^n}(d)$, and satisfies  \eqref{moduli} and \eqref{moduli2}. Hence 
$[\omega] \in \F_q(d,\P^n)$. 
\end{proposition}

\begin{proof}
If we take into consideration Remark \ref{decompositionqforms}, then it only remains to prove that $\omega$ satisfies the integrability equation 
$i_v(\omega)\wedge d\omega = 0, \,\, \mbox{for all}\,\, v\in \bigwedge^{q-1}\CC^{n+1}$. 
This follows by a straight forward calculation using that
$d\omega= \tfrac{dF}{F}\wedge \omega$
and $i_v(\omega)\wedge \omega = 0$. 
\end{proof}

The next result characterizes in an useful way the condition $i_{\d}(\lambda)=0$ for $\lambda \in \Gr(q,\CC^m)$.

\begin{lemma}
 For $\lambda = \lambda^1\wedge \dots \wedge \lambda^q \in \bigwedge^q( \CC^m)-\{0\}$, the following equations are equivalent:
 \begin{enumerate}[i)]
  \item $i_{\d}(\lambda) = 0$
  \item $i_{\d}(\lambda^i) = 0, \hspace{0.3cm} \forall i=1,\dots,q$.
 \end{enumerate}
\end{lemma}

\begin{proof}
 The equivalence follows from 
 $$i_{\d}(\lambda) = \sum_{j=1}^m (-1)^j \,\, i_{\d}(\lambda^j) \,\, \lambda^1 \wedge \dots \wedge \hat{\lambda}^j \wedge \dots \wedge \lambda^q,$$
 and the fact that  $\{\lambda^1 \wedge \dots \wedge \hat{\lambda}^j \wedge\dots \wedge \lambda^q\}_j$ are linearly independent.
\end{proof}

\begin{corollary}\label{eqgrass}
 The following equality holds:
$$\{\lambda \in \Gr(q,\CC^m): i_{\d}(\lambda) = 0\} = \Gr(q,\CC^m_{\d}),$$
where $\CC^m_{\d}$ denotes the linear space of vectors $\mu \in \CC^m$ such that $i_{\d}\mu = \mu \cdot \d = \sum_{i=1}^m \mu_i d_i =0$.
\end{corollary}

\begin{remark}\label{casogeneral1}
 Notice that we have not proved that if $\omega$ is a logarithmic q-form of type $\d$ as in Definition \ref{pollogformsdef1}, then the moduli equations \eqref{moduli} and \eqref{moduli2} are equivalent to the condition $\lambda \in \Gr(q,\CC^m_{\d})$, i.e., $\omega$ globally decomposable  outside  $\D_F = (F=0)$. Proposition \ref{characterizationvectors} only states a partial answer to that problem. However, when $q=2$ the equivalence is in fact true according to Proposition \ref{pollogformsprop1}. It is also remarkable that the case $q>2$ is an open question in \cite[Problem 1]{CL2}.
\end{remark}

\subsection{The logarithmic varieties and the natural parametrization.}\label{param}
Consider again natural numbers $q\in \NN$, $n\in \NN_{>q+1}$, $d,m\in \NN_{\ge q+1}$ and a partition $\d$ of $d$ into $m$ parts.
According to Proposition \ref{characterizationvectors} we denote by $l_q(\d,n)$ the algebraic set of logarithmic q-forms of type $\d$ which are globally decomposable outside its corresponding divisor $\D_F = (F=0)$, and whose projective class determine a codimension $q$ foliation.  
Actually, this set also coincides with the image of the multi-linear map
\begin{align*}
\phi: (\CC^m_{\d})^q\times \prod_{i=1}^m S_{d_i}&\longrightarrow H^0(\P^n,\Omega^q_{\P^n}(d))\\
\nonumber ((\lambda^1,\dots,\lambda^q),(F_1,\dots,F_m))&\longmapsto \sum_{I:|I|=q} (\lambda^1\wedge \dots \wedge \lambda^q)_{I}\hat{F_{I}} dF_I.
\end{align*}

\begin{definition}\label{defnaturalparametrization} We introduce the following map as our natural parametrization:
\begin{align}\label{naturalparametrization}
& \rho:\PP_q(\d) :=\Gr(q,\CC^m_{\d})\times \prod_{i=1}^m \P(S_{d_i}) \xdashrightarrow{\hspace*{0.8cm}}  \,\,\F_q(d,\P^n)\subset\P H^0(\P^n,\Omega^q_{\P^n}(d))\\
&\nonumber (\lambda=[(\lambda^1\wedge \dots\wedge \lambda^q)],\underline{F}=([F_1],\dots,[F_m]))\longmapsto \hspace{0.3cm}[\omega] = \left [ \sum_{I:|I|=q} \lambda_{I}\hat{F_{I}} dF_I \right ].
\end{align}
\end{definition}

From now on and when there is no confusion, we avoid the notation $[\,\,]$ for the corresponding projective classes of the elements involved in the definition
of $\rho$.

\begin{definition}
We define the logarithmic variety $\L_q(\d,n)$ as the Zariski closure of the image of 
$\rho$, i.e., $\L_q(\d,n) = \overline{\im \rho}\subset \F_q(d,\P^n).$
\end{definition}

\begin{remark} \label{remarklogvarieties}
a) $\rho$ is only a rational map and so it is not well defined on the whole space of parameters $\PP_q(\d)$. In particular, on the parameters which give rise to forms that
vanish completely (base locus of the parametrization). 

\noindent b)The space $\PP_q(\d)$ is an algebraic irreducible projective variety of dimension 
$ \sum_{i=1}^m \binom{n+d_i}{d_i} - m + q(m-1-q).$

\noindent c) Notice that the image of $\rho$ also coincides with 
$\P l_q(\d,n)$. 

\noindent d) The space $\L_q(\d,n)$ is a projective irreducible subvariety of $\F_q(d,\P^n)$. 
\end{remark}

With some conditions among $\d$, we will prove that $\L_2(\d,n)$ is an irreducible component of the space $\F_2(d,\P^n)$. From now on, in general, we assume $q=2$.

For our prompt purposes we need to assume some generic conditions on the parameters in $\PP_2(\d)$.   
The polynomials $\{F_i\}$ and the constants $\{\lambda_{ij}\}$ are general according to the following definition.

\begin{definition}\label{generic}
We take the non-empty algebraic open subset of $\PP_2(\d)$ defined by
\begin{multline*}
\U_2(\d) = \{(\lambda,\underline{F})\in \PP_2(\d) : \,\lambda_{ij}\ne 0, \hspace{0.2cm} \lambda_{ij}  - \lambda_{ik} + \lambda_{jk}  \ne 0,\,   \lambda_{ij}  - \lambda_{ik} - \lambda_{jk}  \ne 0 \\ 
\,\mbox{and} \,
F_1,\dots,F_m \,\,\mbox{are smooth irreducible with normal crossings} \}. \nonumber
 \end{multline*}
 
\end{definition}

\begin{remark}
 It follows from Corollary \ref{jouanoloulemmasnc} that $\U_2(\d)$ does not intersect the base locus of $\rho$.
\end{remark}

A more detailed analysis of the base locus of $\rho$, including a description of its irreducible components and scheme structure will be
pursued in \cite{CGM2}.

\subsection{The Zariski tangent space of $\mathcal{F}_2(d,\mathbb{P}^n)$.}\label{modulizariskitangent}
For a scheme $X$ and a point $x \in X$ we
denote by $\T_x X$ the Zariski tangent space of $X$ at $x$. Remember that for a classical projective space $X= \P\mathbb{V}$, 
associated to a finite dimensional vector space $\mathbb{V}$,  
it is common to identify its Zariski tangent space at a given point $x = \pi(v)$ with
$$\T_{x}X = \mathbb{V}/_{\langle v \rangle}.$$
Also recall that $\F_2(d,\P^n)$ is the closed subscheme of the projective space $\P(H^0(\P^n,\Omega^2_{\P^n}(d)))$ defined by equations \eqref{moduli} 
and \eqref{moduli3}. With a slight abuse of notation, the Zariski tangent space $\T_{\omega}\F_2(d,\P^n)$
can be represented by the forms $\alpha \in H^0(\P^n,\Omega^2_{\P^n}(d))/_{\langle \omega \rangle}$ such that
 \begin{align*}
 (\omega + \epsilon \alpha )\wedge (\omega + \epsilon \alpha )&=0 \hspace{1cm}  \mbox{and} \\
 (i_v\omega+ \epsilon i_v\omega)\wedge (d\omega + \epsilon d\alpha)&=0  \hspace{0.5cm}\hspace{0.2cm}\forall v \in \CC^{n+1}\,\mbox{,}  \,\,\mbox{with} \hspace{0.2cm}\epsilon ^2=0.
 \end{align*}

\begin{remark}
 If we fix an element $\omega \in \F_2(d,\P^n)$, then a simple calculation shows that $\alpha\in  H^0(\P^n,\Omega^2_{\P^n}(d))/_{\langle \omega \rangle}$
 belongs to the Zariski tangent space of $\F_2(d,\P^n)$ at $\omega$ if and only if it fulfills the following equations:
\begin{align}
\alpha \wedge \omega &= 0  \label{perturbation}\\
(i_v\omega \wedge d\alpha) + (i_v\alpha \wedge d\omega)&=0  \hspace{0.5cm}  \forall v \in \CC^{n+1}. \label{perturbation2}
\end{align}
We refer to the first equation as the \textit{decomposability perturbation equation} and to the second one as the \textit{integrability perturbation equation}.
In conclusion, we have
$$\T_{\omega}\F_2(d,\P^n) = \{\alpha \in H^0(\P^n,\Omega^2_{\P^n}(d))/_{\langle \omega \rangle} :  \alpha \,\, \mbox{satisfies} \,\, \eqref{perturbation} \,\, \mbox{and} \,\, \eqref{perturbation2}\}.$$
 \end{remark}

\subsection{The derivative of the natural parametrization.}\label{The derivate of the natural parametrization}
One of our main purposes is to show that the derivative of $\rho$ is surjective (see for instance Section \ref{mainresults}). In order to carry out this plan we need to set some notation.

As it was explained in the previous section, for every $e\in \NN$ we have
$$\T_{\pi(F)}\P(S_e) = S_e/_{\langle F \rangle}.$$ 
From now on we will write $F'$ for a general element of this space.

Moreover, let $\lambda = [\lambda^1\wedge \lambda^2]$ be an element in the grassmannian space $\Gr(2,\CC^m_{\d})$.
The Zariski tangent space of $\Gr(2,\CC^m_{\d})$ at $\lambda$ has a natural identification with the space of antisymmetric vectors 
$\lambda' \in \bigwedge^2 (\CC^m_{\d})/_{\langle \lambda \rangle}$ such that
\begin{equation}\label{tangrasseq}
\lambda' \wedge \lambda = 0.
 \end{equation}
Even more, the vectors $\lambda'$ satisfying the above condition can be written as
\begin{equation}\label{tangrass}
\lambda' = (\lambda')^1\wedge \lambda^2 + \lambda^1 \wedge (\lambda')^2,
\end{equation}
where $(\lambda')^1$,$(\lambda')^2 \in \CC^m_{\d}/_{\langle \lambda^1, \lambda^2 \rangle}$. 
In general, we will write $\lambda'$ for the elements of 
$\T_{\lambda}\Gr(2,\CC^m_{\d})$, i.e., for the vectors $\lambda' \in  \bigwedge^2 (\CC^m_{\d})/_{\langle \lambda \rangle}$ fulfilling equation \eqref{tangrasseq} or equivalently 
of the type \eqref{tangrass}.

\begin{remark}\label{remarkderivateimage}
Let $(\lambda^1\wedge \lambda^2, (F_i)_{i=1}^m)=(\lambda,\underline{F}) \in \PP_2(\d)$ be a fixed parameter in the domain of $\rho$ and write
 $\omega= \rho((\lambda,\underline{F})) = \sum_{i\ne j} \lambda_{ij} \hat{F_{ij}} dF_i\wedge dF_j$.
From the multilinearity of $\phi$ and the definition of $\rho$, we obtain the following formula for the derivative of $\rho$:
\begin{multline}
d\rho(\lambda, \underline{F}): \T_{\lambda}\Gr(2,\CC^{m}_{\d}) \times \prod_{i=1}^m \T_{F_i}\P S_{d_i} \longrightarrow \T_{\omega}\F_2(d,\P^n) \\
(\lambda',(F_1',\dots,F_m')) \longmapsto \sum_{i\ne j} \lambda'_{ij} \hat{F}_{ij} dF_i\wedge dF_j \,  \\ 
+ \sum_{i\ne j\ne k} \lambda_{ij} \hat{F}_{ijk}F_k' dF_i\wedge dF_j + 2\sum_{i\ne j} \lambda_{ij} \hat{F}_{ij} dF_i'\wedge dF_j. 
\label{derivateimage}
\end{multline}
We will use the following notation for the forms in the image of the partial derivatives of $\rho$:
\begin{align*}
 \alpha_1 &= d\rho(\lambda, \underline{F})(\lambda',(0,\dots,0)) \,\,\, \mbox{with}\,\, \lambda'\in \T_{\lambda}\Gr(2,\CC^m_{\d}),  \\
  \alpha_2 &= d\rho(\lambda, \underline{F})(0,(F_1',\dots,F_m')) \,\,\, \mbox{with}\,\, F'_i\in S_{d_i}/_{ \langle F_i \rangle}.  
\end{align*}
\end{remark}

\begin{remark}
 For each $(\lambda,\underline{F})\in \PP_2(\d)$ we have an inclusion of vector spaces
 $\im(d\rho(\lambda,\underline{F})) \subset \T_{\omega}\F_2(d,\P^n)$.
 %= \{\alpha \in H^0(\P^n,\Omega^2_{\P^n}(d))/_{\langle \omega \rangle} :  \alpha \,\, \mbox{satisfies} \,\, \eqref{perturbation} \,\, %\mbox{and} \,\, \eqref{perturbation2}\}.$$
 In particular, the forms like $\alpha_1$ and $\alpha_2$ satisfy the perturbation equations \eqref{perturbation} and \eqref{perturbation2}.
\end{remark}

 Now, we want to distinguish the two introduced partial derivatives. With the notation of Section \ref{basicdefs} (see our definitions in \eqref{filtration}), we will see they vanish on different strata $X_{\D_F}^k$ associated to 
the divisor $\D_F =(F=0)= \bigcup_{i=1}^m X_i=(F_i=0)$. 
Under the assumptions of Definition \ref{generic}, each $X_I$ is a smooth complete intersection
of codimension $|I|$. Thus, the stratum $X^k_{\D_F}$ has codimension $k$ and is singular along
$X^{k+1}_{\D_F}$. We have the following results concerning these spaces.

\begin{proposition}\label{generator}
For every $m>1$ and $k\in \{1,\dots,m\}$, the saturated homogeneous ideal $I^{(k)}$ associated to $X_{\D_F}^k$ is generated by $\{ \hat{F_J} \}_{|J|=k-1}.$ Hence we have a surjective map $\bigoplus_{|J|=k-1} \O_{\P^n}(-\hat{d}_J) \twoheadrightarrow \I_{X_{\D_F}^k}$, where $\I_{X_{\D_F}^k} \subset \O_{\P^n}$ denotes the ideal sheaf  of regular functions vanishing on each stratum, and $\hat{d}_J= \sum_{j\notin J}d_j$.
\end{proposition}

\begin{proof}
We will proceed by induction on the total number of polynomials $m$, where $m\ge k$ is required. The base case is trivial, left to the reader. It is also clear that
$I^{(k)}= \bigcap_{J:|J|=k} \langle F_{j_1},\dots,F_{j_k}\rangle,$
and then we are reduced to prove that
$$\bigcap_{J\subset \{1,\dots,m\}:|J|=k}\langle F_{j_1},\dots,F_{j_k}\rangle = \langle \hat{F_J} \rangle_{J\subset \{1,\dots,m\}:\atop|J|=k-1}.$$
One inclusion is always clear and does not require the inductive argument. Next, if we separate in the left term the multi-indexes of the intersection
which do not contain $m$ and use the inductive hypothesis, then we obtain
\begin{equation}\label{eqideal}
I^{(k)} = \bigcap_{J\subset \{1,\dots,m-1\}:\atop|J|=k-1} \langle \hat{F}_{J\cup\{m\}}\rangle \cap \langle F_{j_1},\dots,F_{j_{k-1}},F_m \rangle.
\end{equation}
As a consequence, for every $P \in I^{(k)}$ there exists a family of polynomials $\{H_J\}_{J:|J|=k-1}$ such that
\begin{equation}\label{eqideal2}
P = \sum_{J:|J|=k-1} H_J \hat{F}_{J\cup \{m\}}.
\end{equation}
For each $J_0 = \{j_1,\dots,j_{k-1}\} \subset \{1,\dots,m-1\}$  the class $[P]$ in the quotient ring 
$\bigslant{\CC[z_0, \dots,z_n]}{\langle F_{j_1},\dots,F_{j_{k-1}},F_m\rangle}$ is
equal to zero. This is a consequence of formula \eqref{eqideal}. Hence we have
$[H_{J_0}][\hat{F}_{J_0\cup \{m\}}]=0.$
Finally, since the ring is integral and 
$F_l \notin \langle F_{j_1},\dots,F_{j_{k-1}},F_m\rangle$, for every index $l\notin J_0$ distinct of $m$, we get $[H_{J_0}]=0$. This fact applied to equality \eqref{eqideal2} allows us to show the other needed inclusion.
\end{proof}

\begin{remark}
 If $\omega =\rho(\lambda,\underline{F})$ is a projective logarithmic $q$-form of type $\d$, then $X^{q+1}_{\D_F}$ is contained in its singular set. 
\end{remark}
 
 \begin{remark}\label{remarkint}
Assume  $m>3$. With notation as in Remark \ref{remarkderivateimage}, notice that  $\alpha_1$ vanishes on $X_{\D_F}^3$ and $\alpha_2$ on $X_{\D_F}^4$, and so every element $\alpha$ in the image of $d\rho$ vanishes on $X_{\D_F}^4$.
\end{remark}

\section{Infinitesimal stability of a generic logarithmic 2-form.}\label{infstability2logforms}
\subsection{Main results.}\label{mainresults}

Our desired result can be summarized in the next theorem. For some technical reasons explained later, the vector of degrees $\d$ is assumed to be 2-balanced (see Definition \ref{defbalanced}).

\begin{theorem}\label{main}
Fix natural numbers $n,m\in \NN_{>3}$, and a 2-balanced vector of degrees $\textbf{d}=(d_1,\dots,d_m)$ with $d=\sum_i d_i$. The variety
$\L_2(\textbf{d},n)$ is an irreducible component of the moduli space $\F_2(d,\P^n)$. 
Furthermore, the scheme $\F_2(d,\P^n)$ is generically reduced along this component, in particular at the points of $\rho(\U_2(\d)$).
\end{theorem}

The proof of the above theorem is implied by the surjectivity of the derivative of the natural parametrization 
$\rho$ (Proposition \ref{surjectivity} below), combined with some arguments of scheme theory.
This method is the same as the one used to prove \cite[Theorem 2.1]{CPV} and 
\cite[Theorem 1]{CP}. Moreover, it was also used in \cite[Theorem 8.2]{CGM} for an alternative algebraic proof of the stability of projective logarithmic one forms. 
As a consequence, we will only present a proof of the corresponding proposition.

\begin{proposition}\label{surjectivity}
 With the notation of Theorem \ref{main}, let 
 $(\lambda,\underline{F}) \in \U_2(\d)$ and 
 $\omega  = \rho(\lambda,\underline{F})$. Then the derivative
  $$d\rho(\lambda, \underline{F}): \T_{\lambda}\Gr(2,\CC_{\d}^{m}) \times \prod_{i=1}^m \T_{F_i}\P S_{d_i} \longrightarrow \T_{\omega}\F_2(d,\P^n) $$
  is surjective.
\end{proposition}

\begin{proof}
 It will be attained in the following section.
\end{proof}

\begin{remark}
 Most of our definitions and constructions concerning logarithmic q-forms of type $\d =(d_1,\dots,d_m)$ assume $m \ge q+1$. However, the techniques applied in the proof of Proposition \ref{surjectivity} (and hence of Theorem \ref{main}) require $q=2$ and $m > 3$. Notice that the case $m= q+1$ corresponds to rational q-forms, and it is already known that they determine irreducible and generically reduced components of $\F_q(d,\P^n)$ (see \cite{CPV}).
\end{remark}

\subsection{Surjectivity of the derivative of the natural parametrization.}\label{surjectivitysection}

Let us now start with several steps towards the proof of Proposition \ref{surjectivity}.

\subsubsection{Lemmas.} \label{lemmassection}

We keep the generic conditions assumed in Definition \ref{generic}. In particular, we fix homogeneous irreducible polynomials $F_1,\dots, F_m$ with smooth normal crossings.

\begin{remark}
With the normal crossing assumption, the next property is satisfied:

\medskip
$(\ast)$ \textit{for  $I=\{i_1,\dots,i_k\} \subset \{1,\dots,m\}$ and every point $x \in  \C(X_I)-\{0\}  \subset \CC^{n+1} $, we have
$$d_xF_{i_1}\wedge \dots \wedge d_xF_{i_k} \ne 0,$$}
where $\C(X_I)$ denotes the affine cone over $X_I$ and  $d_xF_j$ is the differential of $F_j$ at $x$. 
\end{remark}
  
The following lemma is similar to that described in \cite[Lemma 2.2]{CPV}, and also it is a consequence of 
Saito's Lemma \cite{SA} turned up for our purposes.

\begin{lemma}[Division Lemma]\label{divisionlemma}
Assume $n,m\in \NN_{\ge 3}$, and fix $k\in \{1,2\}$, integers $j$ and $q$ with $1\le j \le q \le n-2$, and $I = \{i_1,\dots,i_q\}\subset \{1,\dots,m\}$. If $\mu \in H^0(\P^n,\Omega^k_{\P^n}(d))$ is a twisted $k$-form of total degree $d$ such that in homogeneous coordinates satisfies
 \begin{equation}\label{muequation}
(\mu \wedge dF_{i_1}\wedge \dots \wedge dF_{i_j})|_{\C(X_I)}  = 0,
\end{equation}
then there exists forms $\{\gamma_r\}_{r=1}^j \subset  H^0(\CC^{n+1},\Omega^{k-1}_{\CC^{n+1}})$, where each $\gamma_r$ is a homogeneous affine form of total degree 
 $d-d_{i_r}$, such that
 $$\left.\mu\right|_{X_I} = \left (\sum_{r=1}^j \gamma_r \wedge dF_{i_r}\right ) \big{|}_{X_I} .$$
 \end{lemma}

\begin{proof}
Notice that the restricted sheaf $\Omega^1_{\CC^{n+1}}|_{\C(X_I)}$ is an $\O_{\C(X_I)}$-module freely generated on global sections according to: 
$\Omega^1_{\CC^{n+1}}|_{\C(X_I)} = \O_{\C(X_I)} \cdot dz_0|_{\C(X_I)} \oplus \dots \oplus \O_{\C(X_I)} \cdot dz_n|_{\C(X_I)}.$
Moreover, the property  $(\ast)$ of the previous remark implies that the unique singularity of the j-form
$dF_{i_1}|_{\C(X_I)} \wedge \dots \wedge dF_{i_j}|_{\C(X_I)}$
is the point zero. Now, consider the ideal $\A$ generated by the coefficients $\{a_{l_1\dots l_j}\}$ determined by the decomposition
$$(dF_{i_1} \wedge \dots \wedge dF_{i_j})|_{\C(X_I)} = \sum_{0\le l_1 < \dots < l_j \le n} a_{l_1\dots l_j} dz_{l_1}|_{\C(X_I)}\wedge \dots \wedge (dz_{l_j})|_{\C(X_I)}. $$
The depth of $\A$ is greater or equal than three because of the normal crossings hypothesis. 
Hence, we are able to apply Saito's lemma (see \cite{SA}) to divide 1-forms and 2-forms restricted to $\C(X_I)$. From this lemma and equation \eqref{muequation},
we deduce the existence of  homogeneous forms $\tilde{\gamma}_r$ in $H^0(C(X_I),\Omega_{\CC^{n+1}}^{k-1}|_{\C(X_I)})$ of degrees $d-d_{i_r}$ for $r=1,\dots,j$ such that
$$\mu|_{\C(X_I)} = \sum_{r=1}^j \tilde{\gamma}_r \wedge (dF_{i_r}|_{\C(X_I)}).$$
 Next, we need to construct global homogeneous forms $\{\gamma_r\}$ whose restriction to $\C(X_I)$ coincides with $\{\tilde{\gamma}_r\}$. 
When $k=1$, each function $\tilde{\gamma}_r$ can be considered as an element of $H^0(X_I,\O_{X_I}(d-d_{i_r}))$.  
Since each $X_I$ is  a smooth complete intersection, the corresponding restriction map is surjective (see for instance \cite[Exercise 5.5]{H}). Using this fact, we deduce the existence of $\gamma_{k} \in H^0(\P^n,\O_{\P^n}(d-d_{i_k}))$ as wanted.
When $k=2$, we have 
$\tilde{\gamma}_r = \sum_{s=0}^n A^r_s dz_s|_{\C(X_I)},$
where $\{A^r_s\}_{r,s}$ are elements in  $H^0(X_I,\O_{X_I}(d-d_{i_r}-1))$. Finally,
we may apply the previous procedure to each function $A^r_s$, and hence construct homogeneous global affine 1-forms and as claimed.
\end{proof}

\begin{remark}\label{depthcondition}

a)The above proof does not use the fact that the form $\mu$ descends to the projective space, that is $i_R(\mu)=0$. So the conclusion of Lemma \ref{divisionlemma} holds for every homogeneous affine form $\mu\in H^0(\CC^{n+1},\Omega^k_{\CC^{n+1}})$ satisfying equation \eqref{muequation}.

\noindent b) The depth of the ideal $\A$ defined in the previous proof is greater or equal than the dimension
of $\C(X_I)$, which is exactly $n+1-q$.  Hence by an inductive argument, the previous result is in fact true for $k\le  n-q$.
\end{remark}

\begin{remark}
 Let $X$ be an algebraic variety and $Y \xhookrightarrow{i} X$ a subvariety whose sheaf of ideals is denoted by $\mathcal{I}_Y\subset \O_X$. Also let $\mathcal{E}$ be a locally free sheaf on $X$. Taking global sections on the 
 exact sequence
 $$0 \rightarrow \mathcal{E} \otimes \I_Y \rightarrow \mathcal{E} \rightarrow i_*(\mathcal{E}\otimes \O_Y) := \mathcal{E}|_Y \rightarrow 0,$$
 we can identify the global sections of $\mathcal{E}$ vanishing on $Y$ with the elements of $H^0(X,\mathcal{E}\otimes \I_Y)$.
\end{remark}

\begin{lemma}[Vanishing Lemma]\label{vanishinglemma}
Let $\mu \in H^0(\CC^{n+1},\Omega^2_{\CC^n+1})$ be a homogeneous affine 2-form of total degree $d$. If $\mu$ vanishes on the affine cone $\C(X_{\D_F}^k)$ 
for some $k\in \NN_{\ge2}$, then 
$$\mu = \sum_{I:|I|=k-1} \hat{F}_{I} \,\mu_{I}$$
for some homogeneous affine 2-forms $\mu_I$ of  degree $\sum_{i\in I}d_i$.
 \end{lemma}

\begin{proof}
The result is a consequence of Proposition \ref{generator} and the previous remark applied to $X =\CC^{n+1}$, $Y= \C(X_{\D_F}^k)$ and 
$\mathcal{E} = \Omega^2_{\CC^{n+1}}$.
\end{proof}

\begin{lemma}[Fundamental Lemma]\label{fundamentallemma}
Assume $m\in \NN_{>3}$. Consider a family of homogeneous polynomials $\{B_{ijk}\}$ whose indexes are selected on the set $\{1,\dots,m\}$. Suppose $B_{ijk}= B_{jik}$ 
and $\mbox{deg}(B_{ijk})=d_k$. If these polynomials also satisfy the relations
\begin{equation}\label{relations}
B_{ijl}(x) = B_{ikl}(x) = B_{jkl}(x) \hspace{0.5cm}\forall x\in X_{ijkl},
\end{equation}
then for each  $k\in \{1,\dots,m\}$ there exists a polynomial $F_k' \in S_{d_k}$ such that
$$B_{ijk}(x) = F_k'(x)  \hspace{0.5cm}\forall x\in X_{ijk}.$$
\end{lemma}

\begin{proof}
Fix two indexes $j,k \in \{1,\dots,m\}$, write $X= X_{jk}$ and $\D_X$ for the restriction of the divisor $\D_F$ to $X$. 
The family $\{B_{ijk}\}_{i\ne j,\,k}$  determines a well defined object in $H^0(\D_X,\O_{\D_X}(d_k))$, by \eqref{relations}. 
Now, take into consideration the exact sequence
$$0\rightarrow \I_{\D_X}(d_k) \rightarrow \O_X(d_k) \rightarrow i_*(\O_{\D_X})(d_k) \rightarrow 0,$$
and the associated long exact sequence in cohomology
\begin{multline*}
0 \rightarrow H^0(X,\I_{\D_X}(d_k)) \rightarrow H^0(X,\O_{X}(d_k))\rightarrow  H^0(\D_X,\O_{\D_X}(d_k)) \\  
\rightarrow H^1(X,\I_{\D_X}(d_k))\rightarrow \dots 
\end{multline*}
Observe that $\D_X = (\hat{F}_{jk}|_X=0)$, and so
$\I_{\D_X}(d_k) \simeq \O_X(2d_k +d_j -d).$ Moreover,
since $H^1(X, \O_X(2d_k +d_j-d)) = 0$ (see again \cite[Exercise 5.5]{H}), the second morphism of the long sequence is surjective. 
Summarily, there exists an element $\tilde{H}_{jk} \in H^0(X,\O_X(d_k))$ whose restriction to $\D_X$ coincides with $\{B_{ijk}\}_{i \ne j,\,k}$. 

With a similar argument, the restriction map given by
$H^0(\P^n,\O_{\P^n}(d_k)) \xrightarrow{|_X} H^0(X,\O_X(d_k))$
is also surjective. Then, we can choose a global homogeneous polynomial $H_{jk}$ such that
$H_{jk} = B_{ijk}$ on $X_{ijk}$ for all $i \ne j,\,k$.
In addition, we have $H_{jk}=B_{ijk}=B_{jik} = H_{ik}$ on $X_{ijk}$. Hence, we have construct a family of homogeneous polynomials $\{H_{jk}\}_{j,k}$ with $deg(H_{jk})=d_k$ 
such that 
$H_{ik}=H_{jk}$ on every triple intersection $X_{ijk}$. 

Finally, we need to apply again the previous procedure but now over the family $\{H_{jk}\}$  
on the variety $X = X_{k}$ (see also \cite[Proposition 8.9]{CGM}).
From this result we deduce the existence of $F'_k \in S_{d_k}$ such that
$H_{jk} = F'_k  \hspace{0.3cm} \mbox{on} \,\, X_{jk}$ as claimed.
 \end{proof}

 \begin{remark}
  The previous lemma is also true for families of polynomials supported on an arbitrary number of indexes,
  and the proof can be preformed by induction. 
 \end{remark}
 
We end this section with the construction of certain rational vector fields dual, in some sense, to the homogeneous affine 1-forms $\{dF_i\}_{i=1}^m$.

\begin{remark}[Fields with the $\delta$-property]\label{field} 
 
 Fix an integer $r<n$ and select $I = \{i_1,\dots,i_r\} \subset \{1,\dots,m\}$. For each $j\in \{1,\dots,r\}$ we 
 can select, locally on the points of $X_I$, a rational vector field $V^{I}_{j}$ such that
 \begin{equation}\label{deltacondition}
 i_{ V^{I}_{j}}(dF_{i_k}) = \delta_{jk} \hspace{0.2cm} \mbox{on}\,\, X_I \hspace{0.2cm} \forall k \in \{1,\dots,r\}.
 \end{equation}
A similar construction was used in \cite[Proof of Proposition 3.1]{CPV}.
 
The correct objects to consider are the so called \textit{logarithmic vector fields}. The formal definition of a logarithmic derivation along an hypersurface $\D$  of a complex algebraic variety $X$ (with associated ideal sheaf $\I$) is the following:
 $$Der_X(\log \D)_{p} = \{\chi \in (\ensuremath{Der}_X)_p: \chi(\I_p) \subset \I_p \}.$$
We refer to \cite{SA2} for more details and properties. In particular, we shall use that 
 $$\Omega^1_X(\log \D)_p \hspace{0.2cm}\mbox{and} \hspace{0.2cm} \ensuremath{Der}_X(\log \D)_p$$
are reflexive $\O_{X,p}$-modules, dual to each other. Also,  when the divisor has simple normal crossings we are able to apply Saito's criterion (\cite{SA2} p. 270). If we fix local coordinates $f_1, \dots, f_n$
and assume that the divisor is defined by the zero locus of $f_1 \dots f_s$, then $\tfrac{df_1}{f_1},\dots,\tfrac{df_s}{f_s},df_{s+1},\dots,df_n$ is a free system of generators
of $\Omega_X^1(\log \D)_p$. And the local fields
$f_1\cdot \tfrac{\partial}{\partial f_1}, \dots,  f_s \cdot \tfrac{\partial}{\partial f_s}, \tfrac{\partial}{\partial f_{s+1}},\dots,  \tfrac{\partial}{\partial f_{n}}  $
determines a dual basis of $\ensuremath{Der}_X(\log \D)_p$.

In conclusion, for the simple normal crossings projective divisor $\D_F$ determined by the zero locus of $\prod_{i=1}^m F_i,$ 
it is possible to construct vector fields $V^{I}_{j}$ with the property \eqref{deltacondition} as claimed. 
\end{remark}

\subsubsection{Beginning of the proof of Proposition \ref{surjectivity}. Steps 1 to 4.} 

\noindent Note: The hypothesis of $\d$ being 2-balanced will be introduced only when necessary.
\smallskip

As before, fix $n, m\in \NN_{>3}$  and a partition $\d$ of $d$ into $m$ parts. Let $(\lambda^1\wedge \lambda^2,(F_i)_{i=1}^m)=(\lambda,\underline{F})\in \U_2(\d)$ and 
$\omega = \rho(\lambda,\underline{F}) = \sum_{i\ne j} \lambda_{ij} \hat{F}_{ij} dF_i\wedge dF_j .$
Consider $\alpha \in \T_{\omega}(\F_2(d,\P^n))$, i.e., a projective form $\alpha \in H^0(\P^n,\Omega^2_{\P^n}(d))/_{\langle \omega \rangle}$ 
satisfying equations \eqref{perturbation} and \eqref{perturbation2}. We want to prove that there exists an element 
$(\lambda',(F_i')_{i=1}^m) \in \T_{(\lambda, \underline{F})}\U_2(\d)$ such that
$d\rho(\lambda,\underline{F})(\lambda',(F_i')_{i=1}^m) = \alpha.$
Using formula \eqref{derivateimage}, the problem is equivalent to show that
$$\alpha = \sum_{i \ne j} \lambda'_{ij} \hat{F}_{ij} dF_i\wedge dF_j \, + 
							      \sum_{i\ne j\ne k} \lambda_{ij} \hat{F}_{ijk}F_k' dF_i\wedge dF_j +
							      2\sum_{i\ne j} \lambda_{ij} \hat{F}_{ij} dF_i'\wedge dF_j.$$
							      
In the sake of clarity, we will separate the proof in several steps related to the possible
vanishing of $\alpha$ on each stratum $X^k_{\D_F}$.
First, we shall prove that $\alpha$ vanishes on $X^4_{\D_F}$, and describe its restriction to $X^3_{\D_F}$.

\begin{proposition}[Step 1]\label{step1}
If $\alpha$ is a Zariski tangent vector of $\F_2(d,\P^n)$ at $\omega$, then $\alpha|_{X^4_{\D_F}}=0$. Moreover, the following decomposition holds:
$$ \alpha = \sum_{i\ne j \ne k}\hat{F}_{ijk}A_{ijk}dF_i \wedge dF_j + \varepsilon,$$
for some homogeneous polynomials $\{A_{ijk}\}$ with $A_{ijk} = - A_{jik}$, and $\varepsilon \in H^0(\CC^{n+1},\Omega_{\CC^{n+1}}^2)$  
a homogeneous form of total degree $d$ satisfying $\varepsilon|_{X_{\D_F}^3}=0$. 
\end{proposition}

\begin{proof}
Since $X_{\D_F}^3$ is contained in the singular set of $\omega$,  we have 
$i_v(\omega)|_{X_{ijk}}=0$ for each vector $v$ and every piece $X_{ijk}$ of $X_{\D_F}^3$. Also, for every $i,j$ and $k$ notice that
$$d\omega|_{X_{ijk}} = ((\lambda_{ij} - \lambda_{ik}+\lambda_{jk})\hat{F}_{ijk}dF_i\wedge dF_j \wedge dF_k)|_{X_{ijk}}.$$
Then, the integrability perturbation equation $i_v(\alpha)\wedge d\omega + i_v(\omega)\wedge d\alpha=0$ reduces to
$i_v\alpha\wedge ((\lambda_{ij} - \lambda_{ik}+\lambda_{jk})\hat{F}_{ijk}dF_i\wedge dF_j \wedge dF_k) = 0$ on $X_{ijk}.$
Taking into consideration the definition of $\U_2(\d)$ we get 
\begin{equation}\label{eq1}
i_v\alpha\wedge dF_i\wedge dF_j \wedge dF_k = 0 \hspace{0.5cm} \mbox{on}\hspace{0.2cm} X_{ijk} - X^4_{\D_F}. 
\end{equation}
If we select certain local bases at the points of $X_{ijk}$ as in Remark \ref{field} 
turned up for this particular case, then after some straightforward calculation we obtain
$$\alpha \wedge dF_i \wedge dF_j \wedge dF_k = 0 \hspace{0.5cm}\mbox{on}\hspace{0.2cm} X_{ijk}.$$ 
Now, we are able to apply Lemma \ref{divisionlemma} (division lemma), and deduce that
\begin{equation}\label{firstdecomposition}
\alpha = \gamma_i \wedge dF_i + \gamma_j \wedge dF_j + \gamma_k \wedge dF_k \hspace{0.5cm} \mbox{on} \hspace{0.2cm} X_{ijk},
 \end{equation}
for some homogeneous forms $\gamma_l\in H^0(\CC^{n+1},\Omega^1_{\CC^{n+1}})$ of total degree $d-d_l$. 
Again, we need to use fields with the $\delta$-property developed in  
Remark \ref{field}. In particular, we choose a rational vector field $Y_{ijk}^j$ such that
$$(i_{Y_{ijk}^j}(dF_j))|_{X_{ijk}} = 1 \,\,\hspace{0.3cm}\mbox{and}\hspace{0.3cm} (i_{Y_{ijk}^j}(dF_i))|_{X_{ijk}} = (i_{Y_{ijk}^j}(dF_k))|_{X_{ijk}} = 0.$$
Therefore, equation \eqref{eq1} combined with decomposition \eqref{firstdecomposition} implies that
\begin{align*}
\gamma_j \wedge dF_i\wedge dF_j \wedge dF_k =0, 
\end{align*}
at the points of $X_{ijk}$. We can also permute indexes to deduce the same condition for $\gamma_i$ and $\gamma_k$. 
Applying again the division lemma, we finally achieve the following formula:
\begin{equation}\label{f1step1}
 \alpha = G_{ijk}dF_i\wedge dF_j + G_{jki}dF_j\wedge dF_k + G_{ikj}dF_i\wedge dF_k \hspace{0.5cm}\mbox{on}\hspace{0.2cm}X_{ijk},
\end{equation}
where  $\{G_{rst}\}_{r,s,t}$ are homogeneous polynomials of the correct degree. 

Now, if we fix another index $l$ and 
compare the previous decompositions on their restriction to $X_{ijkl}= X_{ijk} \cap X_{jkl} \cap X_{ikl}$, then we get
\begin{align*}
 \alpha|_{X_{ijkl}} &= (G_{ijk}dF_i\wedge dF_j + G_{jki}dF_j\wedge dF_k + G_{ikj}dF_i\wedge dF_k)|_{X_{ijkl}}\\
&=(G_{jkl}dF_j\wedge dF_k + G_{klj}dF_k\wedge dF_l + G_{jlk}dF_j\wedge dF_l)|_{X_{ijkl}}.
 \end{align*}
Due to the normal crossings hypothesis, the restricted affine forms $dF_i|_{X_{ijkl}}$, $dF_j|_{X_{ijkl}}$, $dF_k|_{X_{ijkl}}$ and $dF_l|_{X_{ijkl}}$ are  linearly independent at every point of $X_{ijkl}$. Hence, we have
$G_{ijk}|_{X_{ijkl}} = 0$ for all $l\ne i,j,k$.
As a consequence, for every selection of indexes $i,j$ and $k$ there exists a homogeneous polynomial $A_{ijk}$ such that $G_{ijk} = \hat{F}_{ijk}A_{ijk}$ on $X_{ijk}$. It is easy to check that both families $\{G_{ijk}\}$ and $\{A_{ijk}\}$ can be selected antisymmetric in the first two indexes.

Finally, notice that $\alpha$ and 
$\sum \hat{F}_{ijk}A_{ijk}dF_i\wedge dF_j $
have the same restriction to each $X_{ijk}$. Hence the form 
$\varepsilon = \alpha - \sum \hat{F}_{ijk}A_{ijk}dF_i\wedge dF_j$ is a homogeneous affine form of total degree $d$ that vanishes on $X_{\D_F}^3$, as claimed.
\end{proof}

\begin{corollary}\label{step1corollary}
 Any Zariski tangent vector $\alpha \in \T_{\omega}\F_2(d,\P^n)$ may be written as
 $$\alpha = \sum_{i\ne j \ne k}\hat{F}_{ijk}A_{ijk}dF_i\wedge dF_j + \sum_{i< j}\hat{F_{ij}}\varepsilon_{ij}$$
for some homogeneous affine 2-forms $\{\varepsilon_{ij}\}$. 
 \end{corollary}

\begin{proof}
 It follows from Proposition \ref{step1} (step 1) and Lemma \ref{vanishinglemma} (vanishing lemma).
\end{proof}

The next step deals with the existence of the expected polynomials "$F'_i$". We would like to obtain  
some equations for the polynomials $\{A_{ijk}\}$ of step 1 to
deduce  that each $A_{ijk}/\lambda_{ij}$ only depends on the index $k$. 

\begin{proposition}[Step 2]\label{step2}
 With the notation of Proposition \ref{step1}, define $B_{ijk}= A_{ijk}/\lambda_{ij}$ for $i\ne j \ne k$.
Then, these new polynomials necessarily satisfy the relations
 $$B_{ijl}(x) = B_{jkl}(x) = B_{ikl}(x) \hspace{0.3cm}\forall x\in X_{ijkl}.$$
\end{proposition}

\begin{proof}
We will fix in some order four indexes $i_0,j_0,k_0$ and $l_0$  to deduce the desired conditions on $
X_{i_0j_0k_0l_0}$. 

First, select $j_0$ and $l_0$, and use  
the construction of Remark \ref{field} to select a local rational vector field $Y_{j_0} := Y_{j_0l_0}^{j_0}$
with the corresponding $\delta$-property.
For simplicity, we will work separately on the two terms of the restricted perturbation equation
\begin{equation}\label{formula}
(i_{Y_{j_0}}(\alpha)\wedge d\omega)|_{X_{j_0l_0}} + (i_{Y_{j_0}}(\omega)\wedge d\alpha)|_{X_{j_0l_0}}=0.
 \end{equation}
From Corollary \ref{step1corollary} we have
$$i_{Y_{j_0}}(\alpha) = \sum_{i\ne j\ne k}2\hat{F}_{ijk}A_{ijk}i_{Y_{j_0}}(dF_i)dF_j + \sum_{i< j}\hat{F}_{ij}i_{Y_{j_0}}(\varepsilon_{ij}).$$
Hence its restriction to $X_{j_0l_0}$ is 
\begin{multline*}  
          \sum_k 2\hat{F}_{j_0l_0k}A_{l_0j_0}dF_{l_0} 
        + \sum_i 2\hat{F}_{ij_0l_0}A_{ij_0l_0}i_{Y_{j_0}}(dF_i)dF_{j_0} \\ + \sum_i 2\hat{F}_{ij_0l_0} A_{il_0j_0} i_{Y_{j_0}}(dF_i)dF_{l_0}
        + \sum_j 2\hat{F}_{j_0jl_0}A_{j_0jl_0}dF_j  \\
        + \, \hat{F}_{j_0l_0} i_{Y_{j_0}}(\varepsilon_{j_0l_0}).
\end{multline*}
Also notice that
\begin{multline*}
(d\omega)|_{X_{j_0l_0}} = (\sum_{i\ne j\ne k}\lambda_{ij}\hat{F}_{ijk}dF_i\wedge dF_j \wedge dF_k)|_{X_{j_0l_0}} \\ = (\chi_{j_0l_0} \wedge dF_{j_0}\wedge dF_{l_0})|_{X_{j_0l_0}}
\end{multline*}
for some homogeneous affine 1-form $\chi_{j_0l_0}$. 
As a consequence, we get the following description of the whole \textit{first term} on its restriction to $X_{j_0l_0}$:
\begin{multline*}
i_{Y_{j_0}}(\alpha)\wedge d\omega = 
 \sum_{i\ne j\ne k \ne r \atop  r\ne j_0,l_0} 2\lambda_{ij}A_{j_0rl_0}\hat{F}_{j_0rl_0}\hat{F}_{ijk}dF_r \wedge dF_i \wedge dF_j \wedge dF_k  \\
 + 2\hat{F}_{j_0l_0}i_{Y_{j_0}}(\varepsilon_{j_0,l_0}) \wedge \chi_{j_0l_0} \wedge dF_{j_0}\wedge dF_{l_0},
\end{multline*}
where $\hat{F}_{j_0rl_0}\hat{F}_{ijk}$ can be replaced  by $\hat{F}_{j_0l_0}\hat{F}_{ijkr}$.

Now, we ought to make the same process on the other term of expression \eqref{formula}. In this case, observe that
\begin{align*}
 (i_{Y_{j_0}}(\omega))|_{X_{j_0l_0}} = (\sum_{i\ne j}2\lambda_{ij}\hat{F}_{ij}i_{Y_{j_0}}(dF_i)\,dF_j)|_{X_{j_0l_0}} = (2\lambda_{j_0l_0}\hat{F}_{j_0l_0}dF_{l_0})|_{X_{j_0l_0}}.
 \end{align*}
After some straight forward calculation  we obtain that  the restriction of $ i_{Y_{j_0}}(\omega) \wedge d\alpha$  to $X_{j_0l_0}$ coincides with the restriction of the form
\begin{multline*}
 2\lambda_{j_0l_0}\hat{F}_{j_0l_0}dF_{l_0} \wedge \bigg ( \sum_{i\ne j\ne k\ne r}A_{ijk}\hat{F}_{ijkr}dF_i \wedge dF_j \wedge dF_r  \\
 + \sum_{i\ne j\ne k}\hat{F}_{ijk}dA_{ijk}\wedge dF_i\wedge dF_j + \sum_{i\ne j\ne k \atop i<j}\hat{F}_{ijk}dF_k \wedge \varepsilon_{ij} + \sum_{i< j}\hat{F}_{ij}d\varepsilon_{ij} \bigg ).
\end{multline*}

Next, we need to add the two term's expressions obtained. Also, notice that we are allowed 
to cancel the polynomial
factor $\hat{F}_{j_0l_0}$. Then, select two more distinct indexes $i_0$ and $k_0$, and further restrict to $X_{i_0j_0k_0l_0}$ to obtain
\begin{multline*}
\left ( (\lambda_{j_0k_0} + \lambda_{k_0l_0})A_{j_0i_0l_0} + (\lambda_{i_0j_0}+ \lambda_{l_0i_0})A_{j_0k_0l_0} + \lambda_{j_0l_0}A_{i_0k_0l_0}\right )  \hat{F}_{i_0j_0k_0l_0} \\ dF_{i_0}\wedge dF_{j_0} \wedge dF_{k_0} \wedge dF_{l_0} = 0. 
\end{multline*}
Due to the normal crossings hypothesis, the definition of the polynomials $\{B_{ijk}\}$ and the order in which the indexes were selected, 
we have produced the following equation  on $X_{i_0j_0k_0l_0}$:
\begin{multline*}
(\lambda_{j_0k_0}\lambda_{j_0i_0} + \lambda_{k_0l_0}\lambda_{j_0i_0})B_{j_0i_0l_0} + (\lambda_{i_0j_0}\lambda_{j_0k_0}+ \lambda_{l_0i_0}\lambda_{j_0k_0})B_{j_0k_0l_0}\\ + \lambda_{j_0l_0}\lambda_{i_0k_0}B_{i_0k_0l_0} =0,
\end{multline*}
from now denoted by [$Eq_{j_0l_0i_0k_0}$].
Since $A_{ijk}=-A_{jik}$, if we consider the sum of equations
[$Eq_{i_0l_0j_0k_0}$], [$2Eq_{j_0l_0i_0k_0}$] and [$Eq_{k_0l_0i_0j_0}$], then we have
$$(\lambda_{i_0j_0} - \lambda_{i_0k_0} + \lambda_{j_0k_0})(B_{i_0j_0l_0}-B_{j_0k_0l_0}) = 0 $$
on $X_{i_0j_0k_0l_0}$ as wanted. Finally, the other equality follows from an appropriated permutation of indexes.
\end{proof}

In the next step, we deduce that any tangent vector at
$\omega$ may be decompose as a sum of a perturbation in the image of $d\rho$ that vanishes on $X_{\D_F}^4$,  
with another tangent vector that vanishes on a stratum of lower codimension. 

\begin{proposition}[Step 3]\label{step3}
For every $\alpha \in \T_{\omega}\F_2(d,\P^n)$,
there exists a family of homogeneous polynomials $F'_{1},\dots,F'_m$ of degrees $d_1,\dots,d_m$ and another tangent vector
$\beta \in \T_{\omega}\F_2(d,\P^n)$
with $\beta|_{X_{\D_F}^3}=0$, such that
\begin{align*}
\alpha &= \sum_{i\ne j\ne k}\lambda_{ij} \hat{F}_{ijk}F_k'dF_i\wedge dF_j + \sum_{i\ne j}2\lambda_{ij} \hat{F}_{ij} dF_i'\wedge dF_j + \beta \\
&= d\rho(\lambda,\underline{F})(0,(F_i')_{i=1}^m) + \beta .
 \end{align*}
\end{proposition}

\begin{proof}
First, we use Propositions \ref{step1} and \ref{step2} (steps 1 and 2) to get the following decomposition for $\alpha$:
$$ \alpha = \sum_{i\ne j \ne k}\lambda_{ij}\hat{F}_{ijk}B_{ijk}dF_i \wedge dF_j + \varepsilon,$$
where $\{B_{ijk}\}$ satisfy the relations
$B_{ijl} = B_{jkl} = B_{ikl}$
on $X_{ijkl}$, and $\varepsilon|_{X^3_{\D_F}}=0$. 
By Lemma \ref{fundamentallemma} (fundamental lemma), 
there exist homogeneous polynomials $F'_1,\dots,F'_m$ of respective degrees $d_1,\dots,d_m$ such that
$B_{ijk} = F'_k $ on $X_{ijk}$. 
Then, it is clear that the forms $\alpha$ and 
$\sum_{i\ne j\ne k} \lambda_{ij}\hat{F}_{ijk}F'_k dF_i \wedge dF_j $
have the same restriction to $X_{\D_F}^3$. 
Moreover, if we add and subtract a suitable term, then we obtain 
\begin{multline*}
\alpha = \sum_{i\ne j\ne k}\lambda_{ij} \hat{F}_{ijk}F_k'dF_i\wedge dF_j + \sum_{i\ne j}2\lambda_{ij} \hat{F}_{ij} dF_i'\wedge dF_j + \beta \\ = d\rho(\lambda,\underline{F})(0,(F_i')_{i=1}^m) + \beta ,
 \end{multline*}
where $\beta$ is a homogeneous affine form such that $\beta|_{X_{\D_F}^3}=0$. 
Finally, since $d\rho(0,(F_i'))$ and $\alpha$ are Zariski tangent vectors at $\omega$, 
the same condition holds for $\beta$. 
\end{proof}

\begin{remark}
From now on, we are reduced to prove our result 
for elements $\beta \in \T_{\omega}\F(d,\P^n)$ with the additional hypothesis $\beta|_{X_{\D_F}^3}=0$. 
In advanced, taking into consideration Remark \ref{remarkderivateimage}, these forms are expected to be related to perturbations of the 
coefficients $\lambda$. This is
going to be true only assuming certain extra condition among $\d$ (balanced case). 
 \end{remark}

Next proposition sets the background to end the proof, and is useful to understand the possible 
trouble in the non balanced case.

By Lemma \ref{vanishinglemma}, if $\beta \in \T_{\omega}\F(d,\P^n)$ vanishes on $X_{\D_F}^3$, then it may be written as
$$\beta = \sum_{i\ne j}\hat{F}_{ij}\beta_{ij},$$
for some homogeneous affine forms $\{\beta_{ij}\}$ such that $\beta_{ij}=\beta_{ji}$. 
Now, we would like to obtain further information on these new forms $\beta_{ij}$.

\begin{proposition}[Step 4]\label{step4}With the notations above, for each selection of $i_0,j_0$ and $k_0$ there
exist  $\lambda'_{i_0j_0}\in \CC$ and 
homogeneous polynomials $B_{i_0k_0}^{i_0j_0}$ and $B_{j_0k_0}^{i_0j_0}$ such that 
$$\beta_{i_0j_0} = \lambda_{i_0j_0}'dF_{i_0}\wedge dF_{j_0} + \hat{F}_{i_0j_0k_0}(B_{i_0k_0}^{i_0j_0}dF_{i_0} \wedge dF_{k_0} + B_{j_0k_0}^{i_0j_0}dF_{j_0} \wedge dF_{k_0})$$
on $X_{i_0j_0k_0}$. 
\end{proposition}

\begin{proof}
We will follow a  similar idea to that used in steps 1 and 2.
 Fix $i_0$ and $j_0$, and use Remark \ref{field} to select a rational vector field $Y_{j_0}=Y_{j_0}^{i_0j_0}$ with the corresponding $\delta$-property.
From the integrability perturbation equation restricted to $X_{i_0j_0}$ we have
\begin{multline*}
%  (i_{Y_{j_0}}\beta \wedge d\omega + i_{Y_{j_0}}\omega \wedge d\beta)|_{X_{i_0j_0}} = 
 \sum_{i\ne j\ne k} 2\lambda_{ij}\hat{F}_{i_0j_0}\hat{F}_{ijk}\,i_{Y_{j_0}}(\beta_{i_0j_0})\wedge dF_i \wedge dF_j \wedge dF_k \\ 
  + \sum_{i\ne j\ne k}2\lambda_{j_0i_0}\hat{F}_{i_0j_0}\hat{F}_{ijk}\,dF_{i_0}\wedge dF_k \wedge \beta_{ij}   
 + 4\lambda_{j_0i_0}\hat{F}_{i_0j_0}\hat{F}_{i_0j_0}dF_{i_0}\wedge d\beta_{i_0j_0} = 0  .                             
\end{multline*}
Afterward, we can remove the factor $\hat{F}_{i_0j_0}$, select another index $k_0$, restrict the equation to $X_{i_0j_0k_0}$ and obtain
\begin{multline}
(\lambda_{i_0j_0}-\lambda_{i_0k_0} +\lambda_{j_0k_0})\hat{F}_{i_0j_0k_0}i_{Y_{j_0}}(\beta_{i_0j_0})
dF_{i_0}\wedge dF_{j_0} \wedge dF_{k_0} \\
+ 2\lambda_{j_0i_0}\hat{F}_{i_0j_0k_0}dF_{i_0}\wedge dF_{k_0} \wedge \beta_{i_0j_0}  
+2\lambda_{j_0i_0}\hat{F}_{i_0j_0k_0}dF_{i_0}\wedge dF_{j_0} \wedge \beta_{i_0k_0} = 0.\label{importanteq}
\end{multline}
If we take the wedge product by $dF_{j_0}|_{X_{i_0j_0k_0}}$, then we deduce 
$$(\beta_{i_0j_0} \wedge dF_{i_0}\wedge dF_{j_0} \wedge dF_{k_0})|_{X_{i_0j_0k_0}} = 0.$$
Also, the same conclusion holds for $\beta_{i_0k_0}$ and $\beta_{j_0k_0}$. From Lemma \ref{divisionlemma} (division lemma) we have
$$(\beta_{i_0j_0})|_{X_{i_0j_0k_0}}= \mu_{k_0}^{i_0j_0k_0} \wedge dF_{k_0} + \mu_{j_0}^{i_0j_0k_0} \wedge dF_{j_0} + \mu_{i_0}^{i_0j_0k_0} \wedge dF_{i_0}$$
for some homogeneous affine forms $\mu_{l}^{i_0j_0k_0}$ of degree $d_{i_0} + d_{j_0} -d_l$.

Now, we want to prove that these new forms are also divisible by  $dF_{i_0}$, $dF_{j_0}$ and $dF_{k_0}$.  
Select a new rational vector field $Z_{j_0} := Y_{j_0}^{i_0j_0k_0}$ with the corresponding $\delta$-property, and replace the above decomposition for $\beta_{i_0j_0}$ into  \eqref{importanteq} to obtain
\begin{align*}
((\lambda_{i_0j_0} -\lambda_{j_0k_0} +\lambda_{i_0k_0}) \mu_{j_0}^{i_0j_0k_0}- 2\lambda_{i_0j_0}\,\mu_{k_0}^{i_0k_0j_0}) \wedge dF_{i_0} \wedge dF_{j_0} \wedge dF_{k_0}= 0. 
\end{align*}
For each $l\in \{i_0,j_0,k_0\}$ define $\gamma_{l}^{i_0j_0k_0} := \mu_{l}^{i_0j_0k_0} \wedge dF_{i_0} \wedge dF_{j_0} \wedge dF_{k_0}$ 
(similarly with $\gamma_{l}^{i_0k_0j_0}$ and $\gamma_{l}^{j_0k_0i_0}$). 
Next, our last equation may be written as
$$Eq(I_{i_0j_0k_0}):\hspace{0.4cm} (\lambda_{i_0j_0} -\lambda_{j_0k_0} +\lambda_{i_0k_0})\gamma_{j_0}^{i_0j_0k_0} -2\lambda_{i_0j_0}\gamma_{k_0}^{i_0k_0j_0} = 0, $$
where the tag  $Eq(I_{i_0j_0k_0})$ refers to the order in which indexes were selected. 
Permuting these indexes, we can construct a linear system of equations in order to deduce our claim. 
In particular, from $Eq(I_{i_0j_0k_0})+Eq(I_{k_0j_0i_0})$ we deduce $ \gamma_{k_0}^{i_0k_0j_0}=\gamma_{j_0}^{i_0j_0k_0}$. Then, using again 
$Eq(I_{i_0j_0k_0})$ we get $\gamma_{j_0}^{i_0j_0k_0} =0$ on $X_{i_0j_0k_0}$. With a similar argument we can also prove the other vanishing conditions for the forms $\gamma_{l}^{i_0j_0k_0}$.

Now, we are able to apply again the division lemma to the forms $\mu_l^{i_0j_0k_0}$, and obtain the following  decomposition 
for the original form $\beta_{i_0j_0}$ on $X_{i_0j_0k_0}$:
\begin{align}\label{decomposition}
 \beta_{i_0j_0} = A_{i_0j_0k_0}^{i_0j_0}dF_{i_0}\wedge dF_{j_0} + A_{i_0k_0j_0}^{i_0j_0}dF_{i_0}\wedge dF_{k_0} + A_{j_0k_0i_0}^{i_0j_0}dF_{j_0}\wedge dF_{k_0}.
\end{align}
If we fix another index $l_0$, then all the possible decompositions for $\beta_{i_0j_0}$ must coincide in the intersection 
$X_{i_0j_0k_0l_0}=X_{i_0j_0k_0}\cap X_{i_0j_0l_0}$. In particular, we have
\begin{multline*}
  A_{i_0j_0k_0}^{i_0j_0}dF_{i_0}\wedge dF_{j_0} + A_{i_0k_0j_0}^{i_0j_0}dF_{i_0}\wedge dF_{k_0} + A_{j_0k_0i_0}^{i_0j_0}dF_{j_0}\wedge dF_{k_0} \\
 =  A_{i_0j_0l_0}^{i_0j_0}dF_{i_0}\wedge dF_{j_0} + A_{i_0l_0j_0}^{i_0j_0}dF_{i_0}\wedge dF_{l_0} + A_{j_0l_0i_0}^{i_0j_0}dF_{j_0}\wedge dF_{l_0}.
\end{multline*}
From the normal crossings hypothesis we deduce
\begin{equation*}\label{deduction}
A_{i_0j_0k_0}^{i_0j_0} = A_{i_0j_0l_0}^{i_0j_0}\hspace{0.3cm} \mbox{and} \hspace{0.3cm} A_{i_0k_0j_0}^{i_0j_0}=A_{j_0k_0l_0}^{i_0j_0}=0 \hspace{0.5cm}\mbox{on}\,\, X_{i_0j_0k_0l_0}.
 \end{equation*}
The first condition implies that $A_{i_0j_0k_0}^{i_0j_0}$ does not depend on $k_0$. Notice that its degree equals to zero, and so we write $\lambda'_{i_0j_0}:=A_{i_0j_0k_0}^{i_0j_0}$. 
On the other hand, from the second condition we have 
$A_{i_0k_0j_0}^{i_0j_0}= \hat{F}_{i_0j_0k_0}B_{i_0k_0}^{i_0j_0} \hspace{0.2cm}\mbox{and}\hspace{0.2cm} A_{j_0k_0i_0}^{i_0j_0}= \hat{F}_{i_0j_0k_0}B_{j_0k_0}^{i_0j_0}$ on 
$X_{i_0j_0k_0},$
for some homogeneous polynomials $B_{i_0k_0}^{i_0j_0}$ and $B_{j_0k_0}^{i_0j_0}$ of degrees
$deg(B_{i_0k_0}^{i_0j_0})= 2d_{j_0}+ d_{i_0} - d $ and $deg(B_{j_0k_0}^{i_0j_0})= 2d_{i_0}+ d_{j_0} - d$.
Finally, our claim follows from \eqref{decomposition}.
 \end{proof}

\subsubsection{The balanced assumption and end of the proof. Steps 5 to 7.}\label{secbalanced} 

We need the following definition in order
to restrict the possible degrees of the polynomials $B_{i_0k_0}^{i_0j_0}$ and $B_{j_0k_0}^{i_0j_0}$ introduced in the previous step.

  \begin{definition}\label{defbalanced}
We say that an m-tuple of degrees $\d = (d_1,\dots,d_m)$ is \textit{k-balanced} if for each $I\subset \{1,\cdots,m\}$ of size $|I|=k$, the following inequality holds:
\begin{equation}\label{ineqbal}
\sum_{i\in I}d_i=d_{i_1}+\dots + d_{i_k} < \sum_{l\notin I}d_l.
 \end{equation}
 \end{definition}
 
\begin{ejem}\label{ejemplobal}
 If all the possible degrees are  equal to 1, i.e., $\d =(1,\dots,1)$, then $\d$ is $k$-balanced if and only if $\,2k < m$.
\end{ejem} 

\begin{remark}
When $k=1$, condition \eqref{ineqbal} is the same that appears in \cite[Corollaries 5.10 and 5.11]{GKZ} and \cite[Definition 8.16]{CGM}. 
\end{remark}

 \begin{remark}\label{balancedremark}
 If a vector $\d$ is $k$-balanced then it is $k'$-balanced for all $k'$ lower than $k$.
Although, the converse of this fact is trivially not true (not even for large $m$). For example:
 $$\d=\big(1,2,\dots,m-2, \tfrac{(m-2)(m-1)}{2},\tfrac{(m-2)(m-1)}{2}\big)$$
is 1-balanced but not 2-balanced for every $m>3$.     
 \end{remark}

Now, we are ready to continue with the following step.

 \begin{proposition}[Step 5]\label{step5}
Suppose $\d$ is 2-balanced. 
Then, for any $\beta \in \T_{\omega}\F_2(d,\P^n)$ that vanishes on $X^3_{\D_F}$ there exist constants $\{\lambda'_{ij}\}_{i\ne j}$ and 
a homogeneous affine form $\gamma \in H^0(\CC^{n+1},\Omega^2_{\CC^{n+1}})$ of total degree $d$, with $\gamma|_{X_{\D_F}^2}=0$, such that 
 $$\beta = \sum_{i\ne j} \lambda_{ij}'\hat{F_{ij}}dF_i \wedge dF_j + \gamma.$$
\end{proposition}

\begin{proof}
From Proposition \ref{step4} (step 4) we know that $\beta = \sum_{i\ne j}\hat{F}_{ij}\beta_{ij}$, where
$$ \beta_{ij} = \lambda_{ij}'dF_{i}\wedge dF_{j} + \hat{F_{ijk}}(B_{ik}^{ij}dF_{i} \wedge dF_{k} + B_{jk}^{ij}dF_{j} \wedge dF_{k})\hspace{0.3cm}\mbox{on}\,\, X_{ijk}.$$
Since $\d$ is 2-balanced, computing degrees we deduce that the polynomials $B_{ik}^{ij}$ and $B_{jk}^{ij}$ must 
be equal to zero. Then, for every $i,j$ and $k$ we have
\begin{equation}\label{eqstep5}
\beta_{ij}= \lambda_{ij}'dF_i \wedge dF_j \hspace{0.5cm}\mbox{on}\hspace{0.2cm}X_{ijk}.
 \end{equation}
Observe that the form $(\beta_{ij} - \lambda_{ij}'dF_i \wedge dF_j)|_{X_{ij}} $ vanishes on the divisor 
$(\hat{F}_{ij}|_{X_{ij}}=0)$. Hence, with a slight modification of the vanishing lemma (Lemma \ref{vanishinglemma}), we get
$\beta_{ij} - \lambda_{ij}'dF_i \wedge dF_j  = \hat{F}_{ij}\mu_{ij}$ on $X_{ij}$,
for some $\mu_{ij}\in H^0(\CC^{n+1},\Omega^2_{\CC^{n+1}})$ homogeneous form of degree $d_i + d_j - \sum_{k\ne i,j}d_k$. Using again our hypothesis among $\d$,
we deduce that equality \eqref{eqstep5} holds in $X_{ij}$.

Finally, observe that $\beta$ and 
$\sum_{i\ne j} \lambda_{ij}\hat{F}_{ij}dF_i\wedge dF_j$
have the same restriction to $X_{\D_F}^2$, and hence their difference vanishes on this stratum.
\end{proof}

\begin{remark}
With the notations above, notice that the forms $\beta$ and $\sum_{i\ne j} \lambda_{ij}'\hat{F_{ij}}dF_i \wedge dF_j$ satisfy the integrability perturbation equation 
\eqref{perturbation2}, and therefore the same holds for $\gamma$. But, a priori, we can not assume that $\gamma$ satisfies the other perturbation equation \eqref{perturbation}.
\end{remark}

\begin{proposition}[Step 6]\label{step6}
Assume $\d$ is 2-balanced. If $\gamma$ is a homogeneous affine 2-form of degree d satisfying 
 the integrability perturbation equation \eqref{perturbation2}
 and $\gamma|_{X_{\D_F}^2}=0$, then $\gamma=0$.
\end{proposition}

\begin{proof}
Using the vanishing lemma, we get
$\gamma = \sum_{l} \hat{F_l}\gamma_l,$
for some homogeneous forms $\gamma_l\in H^0(\CC^{n+1},\Omega_{\CC^{n+1}}^2)$ of total degree $d_l$. 
 
With the same idea as in other steps, we fix $i_0$ and $j_0$ distinct, 
and choose a rational local vector field $Y_{j_0}= Y_{i_0j_0}^{j_0}$ with the corresponding $\delta$-property.
In this case, the restriction of the integrability perturbation equation to $X_{i_0j_0}$ reduces to   
$$\lambda_{j_0i_0}(\hat{F}_{i_0j_0})^2\,dF_{i_0}\wedge dF_{j_0}\wedge \gamma_{i_0} = 0.$$
Then it follows from the division lemma  that 
$$\gamma_{i_0} = \mu_{i_0j_0}\wedge dF_{i_0} + \nu_{i_0j_0}\wedge dF_{j_0} \hspace{0.5cm}\mbox{on}\hspace{0.3cm}X_{i_0j_0},$$
for certain homogeneous affine forms $\mu_{i_0j_0}$ and $\nu_{i_0j_0}$.
But since $dF_{i_0}$ and $\gamma_{i_0}$ have the 
same degree, we deduce that $\mu_{i_0j_0}=0$. 
Now, select a new index $k_0$, and notice that  
$\gamma_{i_0} = \nu_{i_0j_0}\wedge dF_{j_0} = \nu_{i_0k_0}\wedge dF_{k_0}$ on $X_{i_0j_0k_0}$.
If we take the wedge product by $dF_{k_0}$,  then we get
$\nu_{i_0j_0}\wedge dF_{j_0} \wedge dF_{k_0} = 0.$ 
Again by the division lemma, we can select  $A_{i_0j_0k_0}$ such that
$$\gamma_{i_0} = A_{i_0j_0k_0} dF_{j_0} \wedge dF_{k_0} \hspace{0.3cm}\mbox{on}\hspace{0.3cm}X_{i_0j_0k_0}.$$
Furthermore, $A_{i_0j_0k_0}$ is divisible by $\hat{F}_{i_0j_0k_0}$. 
This is just a consequence of the equality
$$ A_{i_0j_0k_0} dF_{j_0} \wedge dF_{k_0} =  A_{i_0j_0l} dF_{j_0} \wedge dF_{l}\hspace{0.3cm}\forall l\ne i_0,j_0,k_0$$
and the normal crossings hypothesis. Then, we obtain 
$$\gamma_{i_0} = \hat{F}_{i_0j_0k_0}B_{i_0j_0k_0} dF_{j_0} \wedge dF_{k_0}$$ on $X_{i_0j_0k_0}$, for some new homogeneous polynomial $B_{i_0j_0k_0}$
of degree $d_{i_0} - \sum_{l\ne i_0 } d_l.$
But notice that this degree is negative because $\d$ is also 1-balanced (see Remark \ref{balancedremark}). 
As a consequence,
$\gamma_{i}|_{X_{ijk}}  = 0 \hspace{0.2cm}\mbox{for all} \hspace{0.2cm}  i\ne j \ne k.$

Finally, for each index $i$, we deduce that $\gamma_i$ vanishes on the subvariety of $X_{\D_F}^3$ defined by
$\bigcup_{j,k\ne i} X_{ijk}$. It is also clear that its corresponding homogeneous saturated ideal is generated by 
$ \{\hat{F}_{ij}\}_{j\ne i}$. With a slight modification of the vanishing lemma (Lemma \ref{vanishinglemma}), we have
$\gamma_{i} = \sum_{j\ne i} \hat{F}_{ij}\,\gamma_{ij},$ for some homogeneous affine forms $\gamma_{ij}$.
Since $\d$ is 2-balanced, each form $\gamma_i$ must be equal to zero, and hence the claim.
\end{proof}

The next proposition summarizes the last step of the proof.
We finally show that each Zariski tangent vector at $\omega$ vanishing on $X^3_{\D_F}$ lies in the image of $d\rho(\lambda,\underline{F})$.

\begin{proposition}[Step 7]\label{step7}
Suppose $\d$ is 2-balanced, and let $\omega=\rho(\lambda,\underline{F})=\sum \lambda_{ij}
 \hat{F_{ij}}dF_i \wedge dF_j $
 be a logarithmic 2-form of type $\d$, with $\lambda \in \Gr(2,\CC^m_{\d})$. If $\beta$ is a Zariski tangent vector of $\F_2(d,\P^n)$ at $\omega$ and also 
 $\beta|_{X_{\D_F}^3}=0$, then there exists
 $\lambda' \in \T_{\lambda}\Gr(2,\CC^m_{\d})$ such that
 $$\beta = \sum_{i\ne j} \lambda_{ij}'\hat{F}_{ij}dF_i \wedge dF_j = d\rho(\lambda,\underline{F})(\lambda',0).$$
\end{proposition}

\begin{proof}
We can use Propositions \ref{step5} and \ref{step6} (steps 5 and 6) to deduce that any Zariski tangent vector with our hypotheses may be written as
$\beta = \sum_{i\ne j} \lambda_{ij}'\hat{F_{ij}}dF_i \wedge dF_j.$
Notice that $i_R(\beta)=0$ if and only if $i_{\d}(\lambda')=0$. 
Hence it is easy to see that $\lambda'=(\lambda_{ij}')$ can 
be considered as an element of $\bigwedge^2 (\CC^m_{\d})/_{\langle \lambda \rangle}$.
Also, recall from Section \ref{The derivate of the natural parametrization}  that  $\T_{\lambda}\Gr(2,\CC^m_{\d})$ can be represented by vectors  $\lambda' \in \bigwedge^2 (\CC^m_{\d})/_{\langle \lambda \rangle}$ such that $\lambda \wedge \lambda' = 0$.
Now, from the decomposability perturbation equation \eqref{perturbation}, that is
$\beta \wedge \omega = 0$, we get
 $$F^2 \cdot \sum_{i\ne j\ne k\ne l}(\lambda'\wedge\lambda)_{ijkl} \frac{dF_i}{F_i} \wedge \frac{dF_j}{F_j}\wedge \frac{dF_k}{F_k} \wedge \frac{dF_l}{F_l}=0.$$
Then from Corollary \ref{jouanoloulemmasnc} it follows that $\lambda'\wedge \lambda=0$, and this implies our claim.

\end{proof}

\begin{corollary}
Combining Propositions \ref{step3} and \ref{step7} (steps 3 and 7) we conclude
the whole proof of the surjectivity of $d\rho$ (Proposition \ref{surjectivity}), 
which also implies our Theorem \ref{main}.  
\end{corollary}

\begin{corollary}
Assume $m,n \in \NN_{> 4}$. From Example \ref{ejemplobal} we deduce that a generic linear logarithmic 2-form is stable. In other words, 
$\L_2(\d,n)$ determines an irreducible component of $\F_2(d,\P^n)$ when $\d = (1,\dots,1)$. Furthermore, observe that any logarithmic form of type $\d=(d_1,\dots,d_m)$
in $\P^n$ is the pullback of a linear logarithmic form by a quasi-homogeneous rational map.
\end{corollary}

\begin{remark}\label{casogeneral}
We expect that  Theorem \ref{main} could  be also proved for a larger codimension $q$ with $2\le q<n-1$ and $m>q+1$.  It seems that the argument  could be performed using the same proof schema, 
studying the restriction of a Zariski tangent vector to each stratum $X^k_{\D_F}$ for $k\le q+2$.  In fact, it is possible to see that each Zariski tangent vector of $\F_q(d,\P^n)$ at a logarithmic $q$-form of type $\d$ as in Definition \ref{defnaturalparametrization} necessarily vanishes on $X_{\D_F}^{q+2}$. 
However, most of the steps of the proof of Proposition \ref{surjectivity} are quite technical for $q>2$ and require new combinatorial ideas, even if we assume that $\d$ is q-balanced. In addition, the open question stated in Remark \ref{casogeneral1} seems to be  important for  obtaining a  generalized proof. 
\end{remark}

\medskip 
 
\begin{flushleft}
\small
Universidad de Buenos Aires and CONICET.\\
Instituto de Investigaciones Matemáticas Luis A. Santaló (IMAS).\\
Ciudad Universitaria - Pabellón 1. Buenos Aires (1428). ARGENTINA.\\

Javier Nicolás Gargiulo Acea, jgargiulo@dm.uba.ar

\end{flushleft} 
 

\begin{thebibliography}{99}
\scriptsize	

 \bibitem{OM} 
Calvo-Andrade, O. \emph{Irreducible components of the space of holomorphic foliations}. Mathematische Annalen - 299(1) (1994), pp. 751-767.

\bibitem{CL2}
Cerveau, D. \& Lins-Neto, A. \emph{Logarithmic foliations.} arXiv:1803.08894 [math.CV] (23 Mar 2018).

\bibitem{CL} 
Cerveau, D. \& Lins-Neto, A. \emph{Irreducible components of the space of holomorphic foliations of degree two in $\mathbb{CP} (n)$, $n\ge3$}. Annals of Mathematics - 143 (1996), pp. 577-612.
 

\bibitem {CGM2} Cukierman, F., Gargiulo, J. \& Massri, C.
   \emph{Geometry of the base locus for logarithmic forms}. In preparation. 
 
\bibitem{CGM} 
Cukierman, F., Gargiulo, J. \& Massri, C.
\emph{Stability of logarithmic differential one-forms.} 
To appear on Transactions of the American Mathematical Society (2018).

\bibitem {CPV}
Cukierman, F., Pereira, J. V. \& Vainsencher, I. \emph{Stability of foliations induced by rational maps}. Annales de la Faculté des Sciences de Toulouse: Mathématiques - 18(4) (2009), pp. 685-715.
 
 \bibitem{CP}
Cukierman, F. \& Pereira, J. V. \emph{Stability of holomorphic foliations with split tangent sheaf}. American Journal of Mathematics - 130(2) (2008), pp. 413-439.



\bibitem{DEL}
Deligne, P. \emph{Équations différentielles à points singuliers réguliers}. Lecture Notes in Mathematics 163. Springer Verlag (1970).



\bibitem{GKZ} 
Gelfand, I. M., Kapranov, M., \& Zelevinsky, A. \emph{Discriminants, resultants, and multidimensional determinants.} Springer Science \& Business Media (2008).	


\bibitem{H}
Hartshorne, R. \emph{Algebraic geometry}.  Springer Science \& Business Media. Vol. 52 (2013).


\bibitem{J} 
Jouanolou, J. P. \emph{\'{E}quations de Pfaff alg\'{e}briques}. Lecture Notes in Mathematics 708. Springer (1979).


\bibitem{LA}
Lazarsfeld, R. K. \emph{Positivity in algebraic geometry I: Classical setting: line bundles and linear series.}  Springer Science \& Business Media. Vol. 48 (2004).

\bibitem{DM} 
de Medeiros, A. S. \emph{Singular foliations and differential $p$-forms}. Annales de la Faculté des Sciences de Toulouse: Mathématiques - 9(3) (2000), pp. 451-466.


\bibitem{STE}
Peters, C. A. \& Steenbrink, J. H. \emph{Mixed Hodge structures}. A Series of Modern Surveys in Mathematics 52. Springer (2008).


\bibitem{SA}
Saito, K. \emph{On a generalization of de Rham lemma.} Annales de l'institut Fourier. Vol. 26. No. 2 (1976).

\bibitem{SA2}
Saito, K. \emph{Theory of logarithmic differential forms and logarithmic vector fields.} J. Fac. Sci. Univ.
Tokyo 27, (1980) pp. 265–291. 



\end{thebibliography}
\end{document}